\documentclass{article}
\usepackage[usenames,dvipsnames]{xcolor}
\usepackage[mathscr]{euscript}
\usepackage{amsmath, amsfonts, amssymb, amsthm, geometry, graphics, graphicx, multicol, multirow, enumerate, float, floatflt, fullpage, tikz, outlines, url, tabularx}
\usepackage[utf8]{inputenc}
\usetikzlibrary{shapes.multipart, matrix, positioning, arrows}
\usepackage{wasysym} 
\usepackage{mathdots} 
\usepackage{scalerel} 
\usepackage{subcaption} 
\newlength\bshft 
\def\fakebold#1{\ThisStyle{\ooalign{$\SavedStyle#1$\cr%
  \kern-\bshft$\SavedStyle#1$\cr%
  \kern\bshft$\SavedStyle#1$}}}

\newtheorem{thm}{Theorem}
\newtheorem{cor}[thm]{Corollary}
\newtheorem{defn}{Definition}
\newtheorem{prop}[thm]{Proposition}
\newtheorem{lem}[thm]{Lemma}
\newtheorem{conj}[thm]{Conjecture}
\newtheorem{rem}{Remark}
\newtheorem{exa}{Example}

\newcommand{\Z}{\mathbb{Z}}
\newcommand{\Q}{\mathbb{Q}}

\newcommand{\A}{\mathcal{A}}

\newcommand{\al}{\alpha}
\newcommand{\be}{\beta}
\newcommand{\la}{\lambda}

\newcommand{\ta}{\theta}

\newcommand{\qdim}{\operatorname{qdim}}

\newcommand{\rk}{\textnormal{rk }}

\newcommand{\tor}{\textnormal{tor }}

\newcommand{\hspan}{\textnormal{hspan}}
\newcommand{\tspan}{\textnormal{hspan}^t}
\newcommand{\qspan}{\textnormal{qspan}}
\newcommand{\qtspan}{\textnormal{qspan}^t}
\newcommand{\width}{\textnormal{hw}}
\newcommand{\twidth}{\textnormal{hw}^t}

\newcolumntype{P}[1]{>{\centering\arraybackslash}p{#1}} 
\definecolor{forest}{rgb}{0.03, 0.47, 0.19}

\graphicspath{ {./_figures/} }

\title{Patterns in Khovanov link and chromatic graph homology}
\author{Radmila Sazdanovic and Daniel Scofield}
\date{\empty}
\begin{document}

\maketitle

\tableofcontents

\section{Introduction}

At the turn of the century Khovanov introduced a new knot invariant, Khovanov link homology, a homology theory whose graded Euler characteristic is the Jones polynomial \cite{Khov1}. 
The rich structure of Khovanov homology contains topological information such as the Rasmussen $s$-invariant and  spectral sequences that relate it to other link homology theories. Although torsion, especially $\Z_2$ torsion, frequently appears in Khovanov homology, its relations with topological properties of knots are not well understood. 
Shumakovitch conjectured that the Khovanov homology of any link (except for disjoint unions or connect sums of unlinks and Hopf links) has torsion of order 2 \cite{Shum1}. This conjecture has been found true for alternating links (which have only $\Z_2$ torsion), and for many semi-adequate links \cite{AP,PPS,PS}. At the same time, odd torsion of many orders is possible in non-alternating links \cite{FKH, MPS}. 

In 2004, Helme-Guizon and Rong categorified the chromatic polynomial for graphs, using a construction analogous to that of Khovanov homology. There is a partial isomorphism between Khovanov homology of a semi-adequate link $L$ and the chromatic homology of a state graph $G_+(D)$ obtained from a diagram $D$ of $L$. The extent of the isomorphism depends only on the length of the shortest cycle in $G_+(D)$. Chromatic homology over the algebra $\A_2 = \Z[x]/(x^2)$ has only $\Z_2$ torsion, and is equivalent to the chromatic polynomial \cite{LS}. When other polynomial algebras of the form $\A_m = \Z[x]/(x^m)$ are used in the construction, the resulting homologies may be stronger than the chromatic polynomial and may contain torsion of arbitrary order \cite{PPS}.

In Section \ref{PatternsSec}, we improve the bound given in \cite{HPR} for the homological span  of chromatic homology, stating the precise span of $H_{\A_2}(G)$ in terms of combinatorial graph data. In addition, we show that the span of $H_{\A_2}(G)$ increases with the length of the shortest cycle in $G$. We give an example of a family of non-alternating links whose Khovanov homology has arbitrarily large correspondence with chromatic homology.

In Section \ref{AC}, we determine how $H_{\A_2}(G)$ changes when a cycle $P_n$ is attached along a single edge or vertex of $G$. Using these results, we describe torsion in Khovanov homology for several families of alternating 3-strand pretzel links and rational 2-bridge links. We give an explicit formula for the rank of the third chromatic homology group on the top diagonal in Section \ref{Jones} and use this formula to compute the fourth and fourth-ultimate 
coefficients of the Jones polynomial for links with certain diagrams. In Section 6 we show that there are no gaps in the torsion of $H_{\A_2}(G)$ when $G$ is an outerplanar graph.

In Section \ref{AlgebraAm}, we provide a lower bound for the homological span of $H_{\A_m}(G)$ and prove that the homological thickness of $H_{\A_m}(G)$ is determined by $m$ and the number of vertices of $G$. We describe several examples of cochromatic graphs distinguished by chromatic homology over $\A_3$, and show that $H^0_{\A_3}$ can distinguish graphs with the same Tutte polynomial and 2-isomorphism type.

\section*{Acknowledgements}
The authors would like to thank Adam Lowrance for sharing his ideas and expertise, Alex Chandler and Jozef Przytycki for helpful discussions, and the referee for corrections and suggestions. The first author is partially supported by the Simons Collaboration Grant 318086.

\section{Background}

\subsection{Khovanov link homology}

In this section, we review the construction of Khovanov link homology following \cite{DBN} and \cite{Viro}.

\begin{figure}[h]
  \centering
   \includegraphics[scale = 0.6]{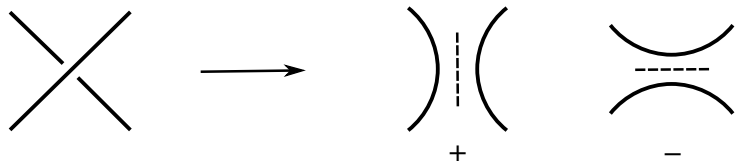}
    \caption{Positive and negative resolutions of a crossing.}\label{res}
\end{figure} 
Let $D$ be a diagram of link $L$. 
The construction of Khovanov homology builds on so-called Kauffman states described in Definition \ref{ks} where each crossing in a diagram $D$ of link $L$ is assigned a choice of a positive or negative resolution, also known as a ``smoothing" of the crossing, Figure \ref{res}.

\begin{defn}\label{ks} 
A Kauffman state of $D$ is a collection of disjoint circles, denoted $D_s$, obtained by resolving each crossing of $D$ in either the positive or negative way according to a function $s:\{\text{crossings of }$D$\} \to \{-1, 1\}.$ An enhanced Kauffman state $S$ is a Kauffman state $s$ in which each circle in $D_s$ is assigned a label $1$ or $x$.  Let $n_+(s)$ denote the number of positive smoothings in Kauffman state $s$, and $n_{-}(s)$ denote the number of negative smoothings.
\end{defn}

Let $\A_2 = \Z[x]/(x^2)$ be the graded $\Z$-module whose generators $1$ and $x$ have degree $1$ and $-1$, respectively.  Order the crossings in an $n$-crossing diagram $D$, and let each Kauffman state be represented by a tuple in $\{0,1\}^n$ with 0s for positive smoothings and 1s for negative smoothings. The $2^n$ Kauffman states of $D$ are in one-to-one correspondence with the vertices of an $n$-dimensional cube: state $s$ corresponds to vertex $\al = (\al_1, \al_2, \ldots, \al_n)$ where $\al_k = 0$ if the $k$th crossing is resolved with a positive smoothing in $s$, and $\al_k = 1$ if it is resolved with a negative smoothing. To the vertex $\al$, we assign the graded $\Z$-module $C_{\al}(D) = \A_2^{\otimes k(s)}$, where $k(s)$ is the number of circles in $s$.

The cochain groups in the Khovanov complex are obtained as direct sums along the diagonals of the cube: $$C^i(D) = \bigoplus_{|\al|=i} C_{\al}(D)$$ where $|\al|$ represents the number of 1s in the label of vertex $\al$. We can think of $C^i(D)$ as a group freely generated by enhanced states of $D$ with $i$ negative smoothings. Let $C^{i,j}(D)$ denote the subgroup of $C^i(D)$ generated by elements whose $\Z$-module grading is $j$.

To define a differential on this cochain complex, we first define maps along the edges of the cube of resolutions. Suppose Kauffman states $s$ and $s'$ only differ at the $k$th crossing, where $s$ has the positive smoothing and $s'$ has the negative smoothing. The corresponding vertices of the cube $\al$ and $\al'$ differ only in the $k$th coordinate, where $\al_k = 0$ and $\al'_k = 1$. Thus there is an edge of the cube from $\al$ to $\al'$, which we denote $e$.  We define the map $d_e:C_{\al}(D) \to C_{\al'}(D)$ as follows. If $s'$ is obtained from $s$ by joining two circles, $d_e$ is the map $m:\A_2 \otimes \A_2 \to \A_2$ that multiplies the labels on those circles. If $s'$ is obtained from $s$ by splitting one circle into two, $d_e$ is the comultiplication map $\Delta:\A_2 \to \A_2 \otimes \A_2$ that sends $1 \mapsto 1 \otimes x + x \otimes 1$ and $x \mapsto x \otimes x$. The differential $d^i:C^{i}(D) \to C^{i+1}(D)$ is defined to be $$d^i = \displaystyle \sum_{\{d_e~:~|\al|=i\}} (-1)^{\xi_e} d_e$$ where $C_{\al}(D)$ is the domain of $d_e$ and $\xi_e \in \{0,1\}$ is chosen as follows.  Suppose the $k$th coordinate of $\al$ is being changed from 0 to 1 along edge $e$ from $\al$ to $\al'$. We let $\xi_e = 1$ if the number of 1s in the set $\{\al_1, \ldots, \al_{k-1}\}$ is odd, and let $\xi_e = 0$ if the number of 1s is even. This assignment ensures that every square face of the cube has a single edge whose associated map has opposite sign from the maps on the other three edges of the square. Since $m$ and $\Delta$ are (co)associative and (co)commutative respectively, each square face anti-commutes, and so $d^2 = 0$.

The chain complex $\mathcal{C}(D) = (C^i(D), d^i)$ is the Khovanov chain complex of $D$. Since the differential preserves degree, $\mathcal{C}(D)$ is a bigraded chain complex. In accordance with the grading conventions found in \cite{DBN}, we shift the original complex by a factor that depends on the number of positive and negative crossings in $D$ (denoted $c_+$ and $c_-$, respectively). The shifted complex is denoted by $\overline{\mathcal{C}}(D) = \mathcal{C}(D)[-c_{-}]\{c_+-2c_-\}$ where $\cdot \{\ell\}$ and  $\cdot [s]$ are the degree and height shift operation given by $\mathcal{C}(D)[s]\{\ell\}^{i,j} = \mathcal{C}(D)^{i-s,j-\ell}.$

The homology of $\overline{\mathcal{C}}(D)$ is denoted $Kh(D)$, the Khovanov homology of diagram $D$. Khovanov homology is a link invariant (\cite{Khov1}, \cite{DBN}) with graded Euler characteristic
$$\chi_q(Kh(L)) = \sum_i (-1)^i \qdim(Kh^i(L)) = \hat{J}(L)$$ where $\hat{J}(L)$ is an unnormalized version of the Jones polynomial of $L$ with $\hat{J}(\ocircle) = q+q^{-1}$, and the graded dimension of a $\Z$-module or a graded vector space $M$ is $\qdim M = \sum_j q^j \dim M^j$ with $M^j$ consisting of homogeneous elements of degree $j$. This polynomial can also be expressed \cite{Kauff1} as a state sum formula
\begin{align} \label{JState}
\hat{J}(L) = (-1)^{c_{-}}q^{c_{+}-2c_{-}}\sum_{i=0}^{c_++c_-} (-1)^i \sum_{\{S~:~n_-(S) = i\}} q^{i}(q+q^{-1})^{|S|}
\end{align}
where $S$ is an enhanced Kauffman state with $|S|$ connected components.

Rational Khovanov homology of alternating links is determined by the Jones polynomial and signature \cite{Lee,Rasmussen}, and the same is true of Khovanov homology with integer coefficients based of the unpublished work of A. Shumakovitch \cite{Shum2}.
For non-alternating links, Khovanov homology is a stronger invariant than the Jones polynomial.

Torsion in Khovanov homology is one source of additional information about knots and links. By far, the most common torsion in Khovanov homology is $\Z_2$. The Khovanov homology of an alternating link (except disjoint unions and connected sums of unknots and Hopf links) has only $\Z_2$ torsion \cite{Shum2}. The Khovanov
homology of a non-alternating link may contain torsion of higher order, including odd torsion \cite{FKH,Katl,MPS}.

\subsection{Chromatic graph homology}

Following the construction given in \cite{HGR}, we describe a homology theory for graphs that categorifies the chromatic polynomial of $G$. Chromatic graph homology construction is analogous to Khovanov homology for links sans comultiplication. 

Let $G=G(V, E)$ be a graph with vertex set $V$ and edge set $E$. The chromatic polynomial $P_G(\la)$ counts the number of ways to color the vertices of $G$ with $\la$ colors, provided that no two adjacent vertices share the same color. 

The chromatic polynomial admits an inclusion-exclusion type formula that plays the same role as the state sum formula Eq.\eqref{JState}  for the Jones polynomial in the construction of Khovanov homology. The precise statement we use is from \cite{HGR} where the so-called {\it state graphs} correspond to the Kauffman states. 
A {\it state graph} denoted by  $S = [G:s]$ is a subgraph of $G$ whose vertex set is $V(G)$ and whose edge set is $s \subseteq E$. The state sum formula for the chromatic polynomial is given by 
\begin{equation} \label{ChState}
P_G(\la) = \sum_{i=0}^{|E|}(-1)^i \sum_{\{s~:~|s|=i\}} \la^{k(s)}
\end{equation}
 where $s$ is a state graph with $|s|$ edges and $k(s)$ connected components.

We label the connected components of state graphs to obtain {\it enhanced state graphs}, analogous to enhanced Kauffman states. When working with graphs, our labels may be generators of any unital, associative algebra $\A$. If we substitute $\la = \qdim~\A$ in the state sum formula (\ref{ChState}), then $P_G(\qdim~  \A)$ can be realized as the Euler characteristic of the homology theory that follows.

Fix an ordering on the edge set $E=\{e_1, e_2, \ldots, e_n\}.$ Analogously to the Khovanov cube of resolutions,  there are $2^{n}$ possible state graphs for $G$ that can be arranged as vertices of an $n$-dimensional cube. Each vertex has a label $(\al_1, \al_2, \ldots, \al_n) \in \{0,1\}^n$, where $\al_k = 1$ if and only if the $k$th edge is present in the corresponding state graph $s$. To each vertex $\al$, we assign the graded $\Z$-module $C_{\A,\al}(G) = \A^{\otimes k(s)}$, where $k(s)$ is the number of connected components in $s$; see Figure \ref{Cube}. Let $C_{\A}^i(G)$ be the group freely generated by enhanced state graphs of $G$ with $i$ edges, and let $C_{\A}^{i,j}(G)$ be the subgroup generated by elements of $C_{\A}^i(G)$ whose $\Z$-module grading is $j$.

Each edge of the cube corresponds to a map $d_e$, defined as follows. Suppose state graphs $s$ and $s'$ are identical except that $s'$ contains the $k$th edge and $s$ does not. The corresponding vertices of the cube $\al$ and $\al'$ differ only in the $k$th coordinate, where $\al_k = 0$ and $\al'_k = 1$. Thus there is an edge of the cube from $\al$ to $\al'$. If the $k$th edge (here denoted $e$) joins different  components of $s$, then $d_e: C_{\A,\al}(G) \to C_{\A,\al'}(G)$ is the map $m:\A \otimes \A \to \A$ that multiplies the labels on these components. If the addition of edge $e$  preserves the number of connected components in $s$, then $d_e$ is the identity map on $\A$.

\begin{figure}[h]
  \centering
   \includegraphics[scale = 0.85]{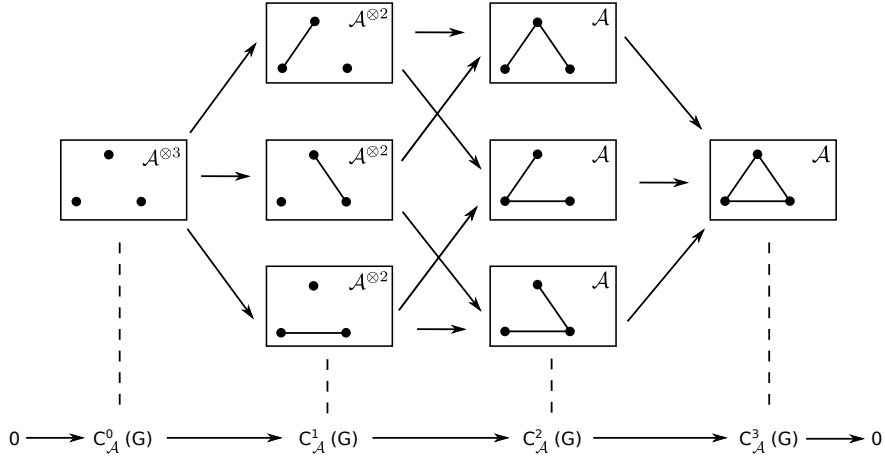}
    \caption{Subgraphs and chromatic chain groups.}
    \label{Cube}
\end{figure}

Chromatic differential $d^i: C^i_{\A}(G) \to C^{i+1}_{\A}(G)$ is defined by 
$d^i = \sum_{\{d_e~:~|\al|= i\}} (-1)^{\xi_e} d_e$ analogously to the construction of the Khovanov differential. The chain complex $C_{\A}(G) = (C_{\A}^i(G), d^i)$ is the chromatic chain complex of $G$. The homology of $C_{\A}(G)$ is denoted $H_{\A}(G)$ and called the chromatic homology of graph $G$.

The graded Euler characteristic of $H_{\A}(G)$ is $\chi_q(H_{\A}(G)) = \sum_i (-1)^i \qdim(H_{\A}^i(G)),$ and since $d^i$ is a degree preserving differential, it recovers the evaluation of the chromatic polynomial at $\lambda=\qdim \A:$ $$\chi_q(H_{\A}(G)) = \chi_q(C_{\A}(G)) =  \sum_{i=0}^{|E|}(-1)^i \sum_{\{s~:~|s|=i\}} (\qdim~ \A)^{k(s)} = P_G(\qdim~  \A)$$
With the special choice of algebra $\A =\Z[x]/(x^2) = \A_2$, this Euler characteristic is $P_G(\qdim~  \A_2) = P_G(q+1)$.

Analogous to the categorification of the Jones polynomial skein relation, the deletion-contraction formula for the chromatic polynomial $$P_G(\la) = P_{G-e}(\la) - P_{G/e}(\la)$$ is categorified by to the short exact sequence \cite{HGR} of chain groups $$0 \to C_{\A}^{i-1,j}(G/e) \to C_{\A}^{i,j}(G) \to C_{\A}^{i,j}(G-e) \to 0$$ which induces a long exact sequence in chromatic homology:
\begin{align}\label{HGRLES}
0 \to H_{\A}^{0,j}(G) \to H_{\A}^{0,j}(G-e) \to H_{\A}^{0,j}(G/e) \to \ldots \to H_{\A}^{i-1,j}(G/e) \to H_{\A}^{i,j}(G) \to H_{\A}^{i,j}(G-e) \to \ldots 
\end{align}

Both $P_G(\la)$ and $H_{\A}(G)$ are trivial if $G$ has a loop, and both remain unchanged if multiple edges are added between two vertices. Therefore, throughout this paper, assume that $G$ is a finite simple graph. For simplicity, we assume $G$ is connected, since \cite[Theorem 3.6]{HGR} provides a formula for chromatic homology of any graph in terms of the chromatic homology of its connected components.

Chromatic graph homology over the algebra $\A_2$  is determined by the chromatic polynomial \cite{CCR,LS}, which is not surprising, since Khovanov homology of alternating knots is almost entirely determined by the Jones polynomial.
\begin{thm} \label{2Tor}\cite[Theorem 1.3]{LS}
The chromatic homology of a graph $H_{\A_2}(G;\Z)$ has only $\Z_2$-torsion.
\end{thm}

\begin{thm}\label{Det}\cite[Theorem 1.4]{LS}
$H_{\A_2}(G;\Z)$ is determined by the chromatic polynomial of $G$. Specifically, $H_{\A_2}(G;\Z)$ consists of a finite number of summands of the form $(\Z \oplus \Z[1]\{-2\} \oplus \Z_2[1]\{-1\})[i]\{v-i\}$ with $i\ge 0$, plus a summand $\Z\{v\} \oplus \Z\{v-1\}$ in homological grading $i=0$ if $G$ is bipartite.
\end{thm}

However, taking the only slightly more complicated algebra $\A_3$ leads to a homology theory which is strictly stronger than the chromatic polynomial and captures different information than the Tutte polynomial \cite{PPS}. In Section \ref{AlgebraAm} we include some results and conjectures about chromatic graph homology for different choices of algebra. 

\subsection{Correspondence between Khovanov and chromatic homology}

For the special choice of algebra $\A_2 = \Z[x]/(x^2)$, the Khovanov link and chromatic graph homology theories admit a partial isomorphism via the graph assigned to a knot and a Kauffman state.

Given a diagram $D$ of a link $L$, let $s_+$ be the Kauffman state of $D$ which has a positive smoothing at each crossing. The graph $G_+(D)$ consists of one vertex for each circle in $s_+$, with an edge connecting any pair of circles related by a crossing in $D$; see Figure \ref{5_1graph}. This construction may also be applied to any other Kauffman state of $G$.

\begin{figure}[h]
  \centering
   \includegraphics[scale = 0.7]{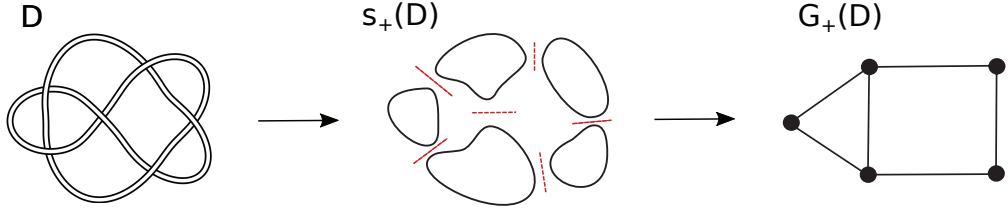}
    \caption{Diagram of $5_1$ and its corresponding planar graph.}
    \label{5_1graph}
\end{figure}

\begin{defn}
The girth of a graph $G$, denoted $\ell(G)$, is the length of the shortest cycle in $G$. We adopt the convention that the girth of a tree is zero, as opposed to considering the girth of a tree to be infinite (see \cite{Boll1, Diestel}).
\end{defn}

\begin{thm}\cite{Przytycki, PS} \label{Correspondence}
Let $D$ be a diagram of link $L$ with $c_-$ negative crossings and $c_+$ positive crossings. Suppose $G_+(D)$ has $v$ vertices and positive girth $\ell$. Let $p = i - c_-$ and $q = v - 2j + c_+ - 2c_-$. For $0 \le i < \ell$ and $j \in \Z$, there is an isomorphism $$H_{\A_2}^{i,j}(G_+(D)) \cong Kh^{p,q}(L).$$ Additionally, for all $j \in \Z$, there is an isomorphism of torsion $\tor H_{\A_2}^{\ell,j}(G_+(D)) \cong \tor Kh^{\ell-c_-,q}(L).$
\end{thm}

Similarities between Khovanov link and chromatic homology go beyond this theorem, and extend mainly to alternating knots and their associated graphs. Note that the following result from \cite {LS} states that the portion of Khovanov homology of any link is the same as Khovanov homology of an alternating link provided that their associated graphs are isomorphic. More precisely, if $D$ is an alternating diagram of a link $L$  and $D'$ is a diagram of any link $L$ such that $G=G_+(D) = G_+(D')$, then we have the following  isomorphism of Khovanov homology groups: 
$Kh^{i,j}(D) \cong Kh^{p,q}(D')$ for $-c_-(D) \le i \le -c_-(D) - \ell(G)-1$ and all $j$ where $p-c_{-}(D_1) = i - c_{-}(D_0)$ and $q+c_+(D') - 2c_-(d') = j+c_+(D) - 2c_-(D)$  \cite[Cor. 5.2]{LS}. 
\begin{defn}
Suppose that bigraded homology $H$ is non-trivial on the set of slope 1 diagonals $\{i+j = a_k\}$ (for chromatic homology) or the set of slope 2 diagonals $\{-2i+j = a_k\}$ (for Khovanov homology). The {\it homological width} of $H$ is $\width(H) = \frac{1}{2}(a_{max}-a_{min}) +1$ where $a_{max}, a_{min}$ are the maximum and minimum values of $a_k$ such that $H^{i,j}$ is non-trivial.

Torsion width of homology is defined analogously, and denoted $\twidth(H)$.
\end{defn}

 In this paper, we focus on chromatic homology over polynomial algebras of the form $\A_m = \Z[x]/(x^m)$.  
 In the case $m=2$, $H_{\A_2}(G)$ is supported on two adjacent diagonals $i+j = v$ and $i+j = v-1$, with torsion on the upper diagonal only \cite{HPR}. In the case that the homological width is equal to $2$, we say that homology is thin. The same is true of Khovanov homology of alternating links \cite{Lee} and a wider class of links, known as {\it thin} links.

\begin{defn}
Let $H$ be either Khovanov or chromatic homology. Let $i_{min}$ be the minimal homological grading with non-trivial homology groups, and let $i_{max}$ be the highest. Then we define the {\it homological span} of homology $H$ as: 
$\hspan(H) = i_{max} - i_{min} + 1.$ The homological span of torsion in $H$, and the {\it quantum} and {\it torsion quantum span} of $H$ are defined similarly and denoted by $\tspan(H), \qspan(H)$ and $\qtspan(H)$ respectively. \end{defn}

While the quantum span of Khovanov homology may be larger than the span of the Jones polynomial (the difference in highest and lowest degree), e.g. $Kh(10_{152})$ \cite{KnotInfo,Katl}, quantum span of chromatic homology corresponds to the span of the chromatic polynomial.  For completeness, we include the following statement about the support of chromatic homology, as we will be improving one of these bounds in Theorem \ref{Span2}.

\begin{prop}\cite[Cor. 13]{HPR}
The chromatic homology of a connected graph $G$ with $v$ vertices is bounded by the following inequalities:\begin{center}
$H^{i,j}_{\A_m}(G) \neq 0 \Rightarrow \begin{cases}
0 \le i \le v-2\\
i+j \ge v-1\\
(m-1)i+j \le (m-1)v
\end{cases}$$\tor H^{i,j}_{\A_m}(G) \neq 0 \Rightarrow \begin{cases}
1 \le i \le v-2\\
i+j \ge v\\
(m-1)i+j \le (m-1)v
\end{cases}$ \end{center}
\end{prop}


\section{Patterns in Khovanov link and chromatic homology}\label{PatternsSec}

In this section we improve the bounds on the span of chromatic homology from \cite{HPR}, which in turn give rather weak lower bounds on the span of Khovanov homology. Finally, we show that as the girth of a graph approaches infinity, the span of Khovanov homology also approaches infinity (Theorem \ref{GirthInfty} and Theorem \ref{KhInfty}). In Section \ref{Gaps} we address more intricate questions about gaps in the support of Khovanov and chromatic homology.

\subsection{Homological span}

In order to compute homological span of chromatic homology we first observe that the minimal quantum grading is equal to the number of blocks in a graph, then define a contracting sequence of graphs that will induce inclusion between their corresponding homology groups.  

A subgraph $B$ is a \textit{block} of $G$ (also known as a biconnected component of $G$) if it is either a bridge or a maximal 2-connected subgraph of $G$ (\cite{Boll1}). We let $b=b(G)$ denote the number of blocks of $G$.

\begin{lem}\label{Jmin}
Let $j_{min}$ be the minimal quantum grading for which $H_{\A_2}^{*,j}(G)$ is non-trivial. Then $j_{min}(H_{\A_2}^{*,j}(G)) = b(G)$.
\end{lem}
\begin{proof}
 Theorem \ref{Det} implies that there is only one non-trivial homology group in the minimal quantum grading: $H_{\A_2}^{v-j_{min}-1,j_{min}}(G)$ on the lower diagonal. Therefore the lowest degree term in the chromatic polynomial $P_G(1+q)$ equals  $\pm \rk H_{\A_2}^{v-j_{min}-1,j_{min}}(G)q^{j_{min}},$ and $q^{j_{min}}$ divides $P_G(1+q).$ In terms of the original variable $\lambda=1+q$ this means that  $j_{min}$ is the multiplicity of the factor $(\la-1)$ in $P_G(\la)$, which is known to be equal to the number of blocks $b$ (\cite{WZ}).
\end{proof}

Given a graph $G$, we define a sequence of graphs obtained by contracting certain edges of $G$. The requirements of Definition \ref{Con}  are tailored to fit conditions in Theorems \ref{Span2} and \ref{D1}, where we use the long exact sequence (\ref{HGRLES}) and results of \cite{CCR} for connected graphs.
In particular, we avoid contracting bridges, as in that case $G-e$ is not connected and its chromatic homology is not thin. 

\begin{figure}[h]
  \centering
   \includegraphics[width=\textwidth]{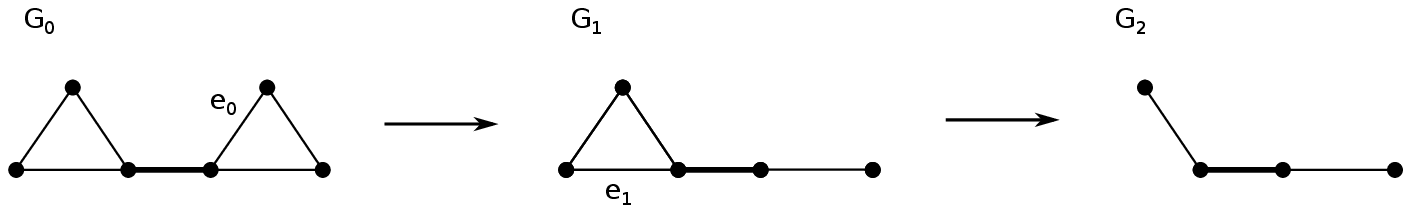}
    \caption{Contraction sequence $\{G_0, G_1, G_2\}$ with ending with a tree $G_2$.}
\end{figure}
\begin{exa}
The contraction sequence shown in Figure 4 reduces the graph 
$G = G_0$ to a tree $G_2$ in $v(G)-b(G)-1 = 2$ steps. Note that bridges (represented by a bold line in $G_0$) can not be contracted, and remain fixed in the contraction sequence. First we reduce the block on the right to a single edge by contracting $e_0$ to obtain the first graph in the contracting sequence $G_1.$ The second contracting step does the same for the block on the left by contracting $e_1$ which reduces $G_1$ to $G_2$ which is a tree.
\end{exa}

\begin{defn} \label{Con}
A \textit{contraction sequence} $G^{/s}$ of a graph $G$ is a set of graphs $G^{/s}=\{G_i\}_{i=0}^n$ such that $G_0 = G$ and each $G_i$ with $0 < i \le n$ is obtained from $G_{i-1}$ by contracting a single non-bridge edge and removing any double edges after the contraction. \end{defn}

\begin{rem}
Note that each contraction decreases the number of vertices in the graph by one; i.e., $v(G_i) = v(G_{i-1})-1$. This procedure can never decrease the number of blocks because contraction of bridges is prohibited, and if  a block has more than two vertices then any edge $e \in E(B)$ is contained in a cycle of $B$, so contraction of $e$ can not eliminate $B$. Moreover, this procedure cannot remove cut-vertices, which implies that each block is contracted separately. 
\end{rem}

\begin{lem} \label{ConLemma} For any graph $G$ there exists a contraction sequence $G^{/s}$ that reduces $G$ to a tree in exactly $v-b-1$ steps, i.e. the first and only tree in a sequence $G^{/s}$ is $\{G_i\}_{i \geq 0}$ is $G_{v-b-1}$.
\end{lem}
\begin{proof}  In the light of Remark 1, we need to prove the existence of the longest possible contracting sequence because after  $v-b-1$  steps we will have a graph with $b+1$ vertices and $b$ blocks so $G_{v-b-1}$  has to be a tree. Theorem 5.12 \cite{Havet} states that if you have a $2$-connected graph (block) B with more than three vertices, there is an edge e of B such that $B/e$  is 2-connected and ensures that we will not get a tree prior to $G_{v-b-1}$.
\end{proof}


The tree obtained in Lemma \ref{ConLemma} is similar to the ``block-cutvertex tree" defined in \cite{HarPrins} (see \cite{Barefoot}).
\begin{thm}\label{Span2} 
For any connected graph $G$ with $v$ vertices and $b$ blocks, $\hspan(H_{\A_2}(G)) = v-b.$
\end{thm}
\begin{proof}
Since $H_{\A_2}^{0, v}(G) = \Z$ for any $G$, it suffices to show that the last nontrivial homology group occurs in homological grading $i = v-b-1$. In particular, we show that the group $H_{\A_2}^{v-b-1, b}(G)$ is non-trivial. Let $G^{/s}$ be a contraction sequence described in Lemma \ref{ConLemma}. If $e_1 \in E(G)$ is the first edge contracted in the sequence, we have a deletion-contraction long exact sequence in chromatic homology Eq. \eqref{HGRLES}:

$$\ldots \to H_{\A_2}^{v-b-2, b}(G-e_1) \to H_{\A_2}^{v-b-2, b}(G/e_1) \stackrel {\alpha_1}{\to} H_{\A_2}^{v-b-1, b}(G) \to \ldots$$
\noindent in which $H_{\A_2}^{v-b-2, b}(G-e_1) \cong 0$ (because $G-e_1$ is connected and has $v$ vertices). Thus map  $\alpha_1$ is injective. Applying the same argument to each of the steps in the contracting sequence yields: 

$$H_{\A_2}^{0,b}(G_{v-b-1})  \lhook\joinrel\xrightarrow {\alpha_{v-b-1}} \ldots \stackrel{\alpha_3}{\hookrightarrow}H_{\A_2}^{v-b-3, b}((G/{e_1})/e_2) \stackrel{\alpha_2}{\hookrightarrow}H_{\A_2}^{v-b-2, b}(G/e_1)  \stackrel{\alpha_1}{\hookrightarrow} H_{\A_2}^{v-b-1, b}(G)$$
with each $\alpha_i$ injective. 

Since $G_{v-b-1}$ is a tree, based on \cite[Example 3.13]{HGR},  $H_{\A_2}^{0,b}(G_{v-b-1})\cong \Z.$  The sequence of injections implies that $H_{\A_2}^{v-b-1, b}(G)$ is also non-trivial. So the span of homology on the $i+j = v-1$ diagonal is at least $v-b$. Since $H_{\A_2}^{v-j-1, j}(G)$ with $j<b$ must be trivial by Lemma \ref{Jmin}, the span is exactly $v-b$.  Theorem \ref{Det} implies that the $i+j = v$ diagonal must have the same homological span. \end{proof}

\begin{thm} \label{D1} 
Chromatic homology $H^{i}_{\A_2}(G)$ contains at least one copy of $\Z$ for each $i$-grading such that $0 \le i \le v-b-1$, that is, $\rk H_{\A_2}^{i,v-i}(G) \oplus H_{\A_2}^{i, v-i-1}(G) > 0.$ 
\end{thm}
\begin{proof}

In case $i=0,1$ the statement follows from  \cite[Thm. 3.1]{PPS}  and  \cite[Lem. 3.1]{PS}. 

Now let $2 \le i \le v-b-1$ and assume that $G$ is not a tree. The proof relies on the contraction sequence of Definition \ref{Con} and the deletion-contraction long exact sequence in chromatic homology. More precisely, we will show that the statement is true for all graphs in the contraction sequence, working backwards starting from $n=v-b-1.$ By Lemma \ref{ConLemma}, there is a contraction sequence $\{G_k\}_{k=0}^{v-b-1}$ of $G$ such that $G_{v-b-1}$ is a tree. We let $G_{v-b-1}$ be our base case, since the result holds for any tree in homological degree zero \cite{HGR}.

Next, assume that the result holds for $G_{k+1}$, $1 \leq k+1 \leq v-b-1$. We show the result also holds for $G_k$.

 In the induction step that follows, $v, E,$ and $b$ refer to the number of vertices, edges, and blocks in $G_k$, respectively. By \cite{CCR}:
\begin{align*}
\rk H_{\A_2}^{i,v-i}(G_k) & = \rk H_{\A_2}^{i-1,v-i}(G_{k+1}) + \rk H_{\A_2}^{i,v-i}(G_k-e)\\
\rk H_{\A_2}^{i,v-i-1}(G_k) & = \rk H_{\A_2}^{i-1,v-i-1}(G_{k+1}) + \rk H_{\A_2}^{i,v-i-1}(G_k-e)
\end{align*}
where $e$ is the edge such that $G_{k+1} = G_k/e$.

Note that $G_{k+1}$ has $v-1$ vertices, $E-1$ edges, and $b$ blocks (the number of blocks cannot change since $e$ was not a bridge). The group $H_{\A_2}^{i-1,v-i}(G_{k+1})$ is on the upper diagonal of the homology of $G_{k+1}$, while $H_{\A_2}^{i-1,v-i-1}(G_{k+1})$ is immediately below it on the lower diagonal. Since $2 \le i \le v-b-1$, we have $1 \le i-1 \le v-b-2$ where $v-b-2 = v(G_{k+1}) - b(G_{k+1}) - 1.$ By assumption, then, $\rk H_{\A_2}^{i-1,v-i}(G_{k+1}) \oplus H_{\A_2}^{i-1,v-i-1}(G_{k+1}) > 0$. This implies that $\rk H_{\A_2}^{i-1,v-i}(G_{k+1}) > 0$ or  $\rk H_{\A_2}^{i-1,v-i-1}(G_{k+1}) > 0$.
\end{proof}


\begin{thm} \label{KhTor} 
Let $D$ be a link diagram of link $L$ whose graph $G_+(D)$ has $v$ vertices, $b$ blocks, and girth $\ell$.

$\tspan(Kh(L)) \ge  hs^t_+=\begin{cases}
v-b-1 & \text{$G_+(D)$ has odd cycle with $\ell \ge v-b-1$}\\
v-b-2 & \text{$G_+(D)$ is bipartite with $\ell \ge v-b-1$}\\
\ell & \text{$G_+(D)$ has odd cycle with $\ell < v-b-1$}\\
\ell-1 & \text{$G_+(D)$ is bipartite with $\ell < v-b-1$}\\
\end{cases}$
\end{thm}

\begin{proof}
The minimal $i$-grading with torsion is either $i=1$ (odd cycle) or $i=2$ (bipartite) \cite{PPS}. On the other hand, $H_{\A_2}(G)$ contains one $\Z_2$ in $(i+1,j-1)$ for each $(i,j), (i+1, j-2)$ knight move pair, based on the  proof of Theorem \ref{Det} \cite{LS}. Therefore, the maximal homological grading with torsion is $i=v-b-1$, where the last $\Z$ occurs.
 If $\ell \ge v-b-1$, the last grading with torsion inside the correspondence is $i = v-b-1$ and the span of torsion is $v-b-1$ (odd cycle) or $v-b-2$ (bipartite). If $\ell < v-b-1$, then the span of torsion is at least $\ell$ (odd cycle) or $\ell-1$ (bipartite).
\end{proof}

\begin{cor} \label{KhTor1} 
Let $D$ be a link diagram whose graphs $G_+(D)$ and $ G_-(D)$ have $v_{\pm}$ vertices, $b_{\pm}$ blocks, and girth $\ell_{\pm}$, respectively. Using notation in Theorem \ref{KhTor} if both  $hs^t_+,  hs^t_->0$ we know that the span of torsion relates to the homological span of Khovanov homology in the following way: $$2 \leq \hspan(Kh(L))-\tspan(Kh(L)) \leq 4.$$
\end{cor}

\subsection{Girth and span}

In this section, we show that as the girth of a graph goes to infinity, so does the span of chromatic homology and also the corresponding part of Khovanov homology. 

\begin{defn}
The girth of link $L$, denoted $gr(L)$, is the maximum value of $\ell(G_+(D))$ over all diagrams $D$ of $L$.
\end{defn}

\begin{lem} \label{blockbound}
Let $M$ be the maximum cycle length in a connected graph $G$. Then $b \le v-M+1$.
\end{lem}

\begin{proof}
Suppose there exists such a graph $G$ with $b > v-M+1$ or, equivalently, $b -1 > v-M.$ By assumption,  there is a cycle of length $M$ in $G$, call it $P_{M}$. The number of vertices in the set $V(G)\setminus V(P_{M})$ is $v-M$, and the number of blocks in $G$ that do not contain $P_{M}$ is $b-1$. Each vertex in $V(G) \setminus V(P_{M})$ can contribute at most one additional block to $G$; i.e., $v-M \ge b-1$. But this contradicts our initial assumption.
\end{proof}

Lemma \ref{blockbound} holds true if we replace $M$ with the length of any cycle in $G$, including the girth. We will use the inequality with $\ell(G)$ to prove Theorem \ref{GirthInfty}.

\begin{thm} \label{GirthInfty}
The homological span of chromatic homology $\hspan(H_{\A_m}(G))$ goes to infinity as the girth $\ell(G)$ goes to infinity.
\end{thm}

\begin{proof} The proof in case $m=2$ follows from Theorem  \ref{Span2} and Lemma \ref{blockbound}. In general, we need Theorem \ref{SpanM}: $\hspan(H_{\A_m}(G)) \ge v-b \ge v-(v-\ell+1) = \ell - 1.$
\end{proof}

\begin{thm}\label{KhInfty}
The homological span of Khovanov homology $\hspan(Kh(L))$ goes to infinity as the girth $gr(L)$ goes to infinity.
\end{thm}
As a corollary, we get that the girth of any link can not be infinite, since we know the span of Khovanov homology. 
\begin{cor} The girth $gr(L)$ of any link $L$ is finite. 
\end{cor} 

\begin{figure}[h]
\begin{subfigure}{0.5\textwidth}
  \centering
   \includegraphics[width=0.43\linewidth]{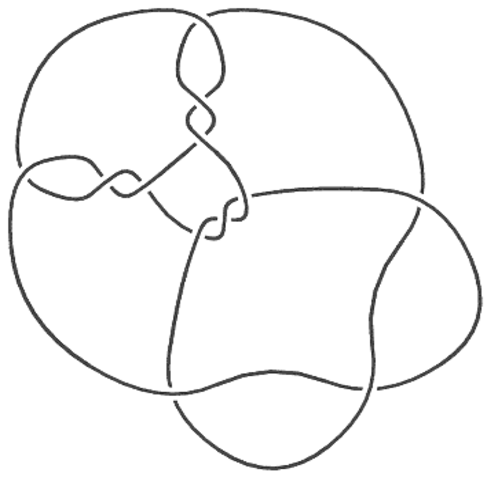}
   \caption{} \label{fig:1a} 
   \end{subfigure}
\begin{subfigure}{0.5\textwidth}
  \centering
   \includegraphics[width=0.43\linewidth]{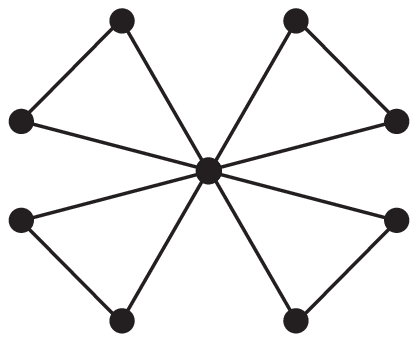}
   \caption{} \label{fig:1b} 
   \end{subfigure}\\
   
   \vspace{8mm}
   \begin{subfigure}{0.5\textwidth}
  \centering
   \includegraphics[width=0.43\linewidth]{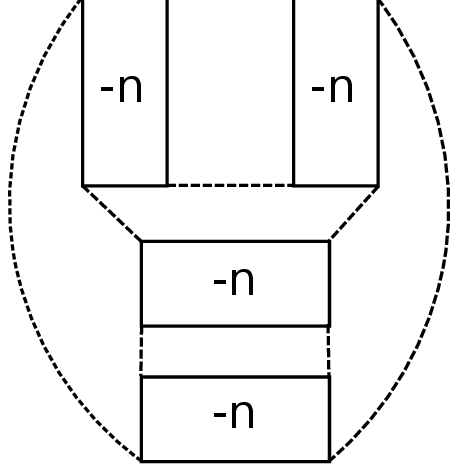}
   \caption{} \label{fig:1c} 
   \end{subfigure}
\begin{subfigure}{0.5\textwidth}
  \centering
   \includegraphics[width=0.43\linewidth]{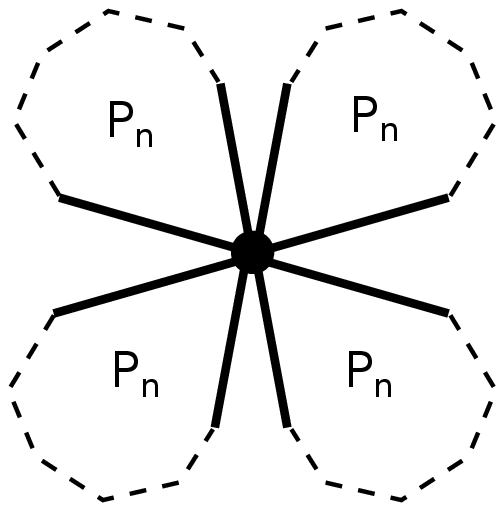}
   \caption{} \label{fig:1d} 
   \end{subfigure}
    \caption{(a) Mirror of the link $12n888$; (b) Graph $G_+(D_3)$ corresponding to diagram in (a); (c) Infinite family $D_n = \overline{-(n;n) (n;n)}$; (d) Graph $G_+(D_n)$ corresponding to diagram in (c)}\label{Mainex}
\end{figure}

On the other hand, Khovanov homology provides an upper bound on girth of a link. More precisely, if Khovanov homology of a knot is thick, the number of non-trivial $i$-gradings before homology becomes thick is the upper bound on girth of $L$ since chromatic homology is always thin. Based on the explicit computations for the first few homological gradings of chromatic homology in \cite{AP, PPS, PS, LS}, if Khovanov and chromatic homology agree on a certain range of gradings, this agreement imposes restrictions on the type of graphs that realize the isomorphism. For example, all such graphs have the same cyclomatic number.

\begin{exa}[Family of links with arbitrarily large girth]
Consider the mirror of the $12$-crossing non-alternating knot $12n888$ \cite{KnotInfo,LS} shown in Figure \ref{Mainex}(a) and denoted $\overline{12n888}$. The Khovanov homology of this knot has minimal homological grading $i=-12$. The homological width of $Kh(\overline{12n888})$ is three but the homology is supported on two diagonals for $-12\leq i <-5$, where the width increases to 3 diagonals. This implies that the girth of $\overline{12n888}$ lies in the range $3 \le gr(\overline{12n888}) \le 7$.

The Conway notation for the standard diagram of $12n888$ is $-(3;3) (3;3)$ \cite{KnotInfo}. Let $D_3=\overline{-(3;3) (3;3)}$ be the diagram corresponding to the mirror of $12n888.$ The graph $G_+(D_3)$ consists of four triangles joined at a single vertex; see Figure \ref{Mainex}(b).

Let $LD_n$ denote a link  determined by diagram $D_n = \overline{-(n;n) (n;n)}$ obtained from $D_3$ by simultaneously increasing the number of twists corresponding to each parameter in Conway symbol \cite{LinKnot}; see Figure \ref{Mainex}(c). The family of graphs associated to these diagrams consists of vertex gluing of four $n$-gons $G_+(D_n) = P_n*P_n*P_n*P_n$; see Figure \ref{Mainex}(d). Thus the girth $\ell(G_+(D_n))=n$  and the range of homological degrees where the isomorphism of Theorem \ref{Correspondence} holds  goes to infinity as $n$ increases. However, the Khovanov homology of these links $LD_n$ is thick with much larger span, and we can only describe a portion of the thin part.  Tables \ref{TKh} and \ref{TCh} contain partial computations for Khovanov homology of $LD_4$  and chromatic homology of $G_+(D_4)=P_4*P_4*P_4*P_4$ with boldface entries denoting matching homology groups.\end{exa}

\begin{table}[H] 
  \centering \small
\renewcommand{\arraystretch}{1}
\begin{tabular}{cc| P{0.6cm} | P{0.6cm} | P{1.1cm} | P{1.1cm} | P{1.3cm} | P{1.3cm} | P{1.3cm} | P{0.3cm} |}
\cline{1-10}
\multicolumn{2}{|c|}{\multirow{2}{*}{\textbf{$Kh^{p,q}(D_4)$}}} & \multicolumn{8}{c|}{\textbf{p}}   \\ \cline{3-10}
\multicolumn{1}{|c}{} & & \textbf{-16} & \textbf{-15} & \textbf{-14} & \textbf{-13} & \textbf{-12} &\textbf{-11} & \textbf{-10} & $\cdots$ \\ \cline{1-10}
\multicolumn{1}{|c|}{\multirow{7}{*}{\textbf{q}}}  & $\vdots$ &  &  &  &  &  & & $\iddots$ & $\iddots$ \\ \cline{2-10}
\multicolumn{1}{|c|}{} & \textbf{-33} &  &  &  &  &  & $\Z^{13}$ & $\Z^{15} \oplus \Z_2^{15}$  &  \\ \cline{2-10}
\multicolumn{1}{|c|}{} & \textbf{-35} &  &  &  &  & $\Z^{10}$ & $\Z^{15} \oplus \Z_2^{13}$ &   &   \\ \cline{2-10}
\multicolumn{1}{|c|}{} & \textbf{-37} &  &  &  & $\fakebold \Z^6$ & $\Z^{13} \oplus \fakebold \Z_2^{10}$ &  &   &  \\ \cline{2-10}
\multicolumn{1}{|c|}{} & \textbf{-39} &  &  & $\fakebold \Z^4$ & $\fakebold \Z^{10} \fakebold \oplus \fakebold \Z_2^6$ &  &   &  &   \\ \cline{2-10}
\multicolumn{1}{|c|}{} & \textbf{-41} &  &  & $\fakebold \Z^6 \fakebold \oplus \fakebold \Z_2^4$ &  &  &   &  &  \\ \cline{2-10}
\multicolumn{1}{|c|}{} & \textbf{-43} & $\fakebold \Z$ & $\fakebold \Z^4$ &  &  &  &   &  & \\ \cline{2-10}
\multicolumn{1}{|c|}{} & \textbf{-45} & $\fakebold \Z$ &  &  &  &  &  &  & \\ \cline{1-10} 
\end{tabular}
\caption{Khovanov homology of the link $LD_4=\overline{-(4;4) (4;4)}$ with boldface entries denoting matching homology with chromatic homology.}\label{TKh}
\end{table} \begin{table}[H] 
  \centering
\renewcommand{\arraystretch}{1}
\begin{tabular}{cc| P{0.6cm} | P{0.6cm} | P{1.1cm} | P{1.2cm} | P{1.3cm} | P{0.6cm} |}
\cline{1-8}
\multicolumn{2}{|c|}{\multirow{2}{*}{\textbf{$H^{i,j}_{\A_2}(G)$}}} & \multicolumn{6}{c|}{\textbf{i}}   \\ \cline{3-8}
\multicolumn{1}{|c}{} & & \textbf{0} & \textbf{1} & \textbf{2} & \textbf{3} & \textbf{4} & $\cdots$ \\ \cline{1-8}
\multicolumn{1}{|c|}{\multirow{7}{*}{\textbf{j}}} & \textbf{13} & $\fakebold \Z$ &  &  &  &  & \\ \cline{2-8}
\multicolumn{1}{|c|}{} & \textbf{12} & $\fakebold \Z$ & $\fakebold \Z^4$ &  &  &  & \\ \cline{2-8}
\multicolumn{1}{|c|}{} & \textbf{11} &  &  & $\fakebold \Z^6 \fakebold \oplus \fakebold \Z_2^4$ &  &  &  \\ \cline{2-8}
\multicolumn{1}{|c|}{} & \textbf{10} &  &  & $\fakebold \Z^4$ & $\fakebold \Z^{10} \fakebold \oplus \fakebold \Z_2^6$ &   & \\ \cline{2-8}
\multicolumn{1}{|c|}{} & \textbf{9} &  &  &  & $\fakebold \Z^6$ & $\Z^9 \oplus \fakebold \Z_2^{10}$ &  \\ \cline{2-8}
\multicolumn{1}{|c|}{} & \textbf{8} &  &  &  &  & $\Z^{10}$ & $\ddots$ \\ \cline{2-8}
\multicolumn{1}{|c|}{} & $\vdots$ &  &  &  &  &   & $\ddots$ \\ \cline{1-8}
\end{tabular}

\caption{Chromatic homology of $G = G_+(D_4)=P_4*P_4*P_4*P_4$ with boldface entries denoting matching homology with chromatic homology.}\label{TCh}

\end{table}



\section{Addition of cycles}\label{AC}

In this section we analyze how attaching a cycle along an edge or vertex affects chromatic homology $H_{\A_2}(G)$ and use these results to describe patterns in Khovanov homology of some alternating $3$-strand pretzel links and rational 2-bridge links.

Recall that the chromatic homology of an $n$-cycle, denoted $P_n$, is determined by the Hochschild homology of the chosen algebra. As mentioned before, we focus on polynomial algebras $\A_m.$ 

\begin{thm}\cite{Przytycki} \label{Polygons}
Let $HH(\A_m)$ be the Hochschild homology of $\A_m$. For $i>0$, Hochschild homology determines the chromatic homology of a cycle graph $P_n$ as follows:

$$HH_{i-n-1,j}(\A_m) \cong H_{\A_m}^{i,j}(P_n)  \cong 
\begin{cases}
\Z_m  & \text{ if } i < n-1, n-i \text{ even, } j = \frac{n-i}{2}m\\
\Z & \text{ if } i < n-1, \lfloor \frac{n-i-1}{2} \rfloor m+1 \le j \le \lfloor \frac{n-i-1}{2} \rfloor m + m - 1\\
0 & \text{ otherwise}\\
\end{cases}$$
\end{thm}

This result, applied to algebra $\A_2$, says the following:

\begin{cor}\label{PolygonLemma}
The chromatic homology for $P_n$ over $\A_2$ is given by\\
$H_{\A_2}^{i,n-i}(P_{n=2k+1})  \cong 
\begin{cases}
\Z_2  & i \text{ odd, } 1 \le i \le n-2\\
\Z & i \text{ even, } 0 \le i \le n-3\\
\end{cases}$ \hspace{0.8cm}$H_{\A_2}^{i,n-i}(P_{n=2k})  \cong 
\begin{cases}
\Z_2  & i \text{ even, } 2 \le i \le n-2\\
\Z & i=0 \text{ or $i$ odd, } 1 \le i \le n-3\\
\end{cases}$\\

\noindent $H_{\A_2}^{i,n-i-1}(P_{n=2k+1})  \cong 
\begin{cases}
\Z  & i \text{ odd, } 1 \le i \le n-2\\
0 & \text{ otherwise}\\
\end{cases}$ \hspace{0.8cm}$H_{\A_2}^{i,n-i-1}(P_{n=2k})  \cong 
\begin{cases}
\Z  & i \text{ even, } 0 \le i \le n-2\\
0 & \text{ otherwise}\\
\end{cases}$
\end{cor}

For other connected graphs, explicit formulae were known only for the first three homological gradings \cite{AP, PPS, PS} and Theorem \ref{4thKh} describes the fourth grading. It is not surprising, but still curious, that these initial gradings in chromatic homology depend only on the bipartiteness and the number of triangles.

\begin{defn} The cyclomatic number $p_1(G)$ of a connected graph $G$ is equal to $p_1(G) = |E|-v+1$.
\end{defn}
\begin{prop} \label{RankDiagonal} \cite{PPS, PS}
Let $G$ be a graph with $t_3$ triangles. Then:

\parbox{0.45\textwidth}{\begin{align*}
&H_{\A_2}^{0,v}(G) = \Z\\
&H_{\A_2}^{0,v-1}(G) =  \begin{cases}
\Z & G \text{ bipartite}\\
0 & \text{ otherwise}
 \end{cases} \end{align*}}
\parbox{0.45\textwidth}{\begin{align*}
&H_{\A_2}^{1,v-1}(G) = \begin{cases}
\Z^{p_1} & G \text{ bipartite}\\
\Z^{p_1-1} \oplus \Z_2 & \text{ otherwise}
\end{cases}\\
&H_{\A_2}^{2,v-2}(G) =  \begin{cases}
\Z^{\binom{p_1}{2}} \oplus \Z_2^{p_1} & G \text{ bipartite}\\
\Z^{\binom{p_1}{2}-t_3+1} \oplus \Z_2^{p_1-1}& \text{ otherwise}
 \end{cases}
 \end{align*}}
\end{prop}

Lemma \ref{Sum} states that entries on the main diagonal in chromatic homology of graph $G$ are determined by entries from the main diagonals of chromatic homology for graphs $G-e$ and $G/e$, provided that edge $e$ is not a bridge.

\begin{lem} \label{Sum}
Given graph $G$ with $v$ vertices and an edge $e \in E(G)$ which is not a bridge, then for all $i \ge 2$,  
\begin{equation}\label{sumEq}H_{\A_2}^{i,v-i}(G) \cong H_{\A_2}^{i-1,v-i}(G/e) \oplus H_{\A_2}^{i,v-i}(G-e)\end{equation} \end{lem}
\begin{proof} The free part of $H_{\A_2}^{i,v-i}(G)$ is determined by computation of rational chromatic homology \cite[Corollary 4.2]{CCR}:
$$
\rk H_{\A_2}^{i,v-i}(G; \Q) = \rk H_{\A_2}^{i-1,v-i}(G/e; \Q) + \rk H_{\A_2}^{i,v-i}(G-e; \Q)$$
The same result, applied in the previous homological grading
$$\rk H_{\A_2}^{i-1,v-i+1}(G; \Q) = \rk H_{\A_2}^{i-2,v-i+1}(G/e; \Q) + \rk H_{\A_2}^{i-1,v-i+1}(G-e; \Q)$$
together with Theorem \ref{Det} determine the torsion on the main diagonal: 
 $$\tor H_{\A_2}^{i,v-i}(G) = \tor H_{\A_2}^{i-1,v-i}(G/e) \oplus \tor H_{\A_2}^{i,v-i}(G-e).$$ \end{proof}

\subsection{Edge gluing of a cycle} \label{edgeGlue}

In this section we analyze how attaching a cycle along an edge or vertex affects chromatic homology $H_{\A_2}(G)$.

We use the notation $G_1|G_2$ to represent the graph obtained by gluing $G_1$ and $G_2$ along a single edge, and $G_1|^kG_2$ for a gluing along $k$ edges, Figure \ref{gluegraphs}. Similarly, $G_1*G_2$ is the gluing of $G_1$ and $G_2$ at a single vertex.

\begin{figure}[h]
  \centering
   \includegraphics[scale = 0.6]{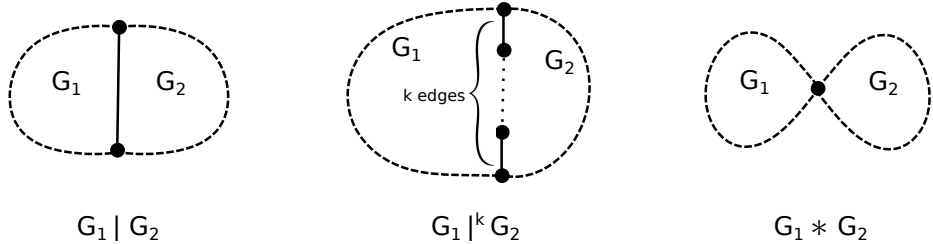}
    \caption{Edge and vertex gluings of graphs.}
    \label{gluegraphs}
\end{figure}

Theorem \ref{PolyEdge} provides an explicit formula for the upper diagonal $i+j=v$ of $H_{A_2}(G|P_n)$, and, together with Theorem \ref{Det}, determines the rest of chromatic homology, i.e. the lower diagonal.

\begin{thm} \label{PolyEdge}
Let $G$ be a graph with $v$ vertices, $E$ edges, and $S_t(G) = \bigoplus\limits_{k=0}^{t} H_{\A_2}^{i-k, v-i+k}(G)$. For $n \ge 3$, 
$$H_{\A_2}^{i,(v+n-2)-i}(G|P_n) \cong
\begin{cases}
S_{n-2}(G) & i > n-2\\
\Z^{E-v+2} \oplus  S_{i-2}(G) & i \le n-2, n-i \text{ odd, $G$ bipartite}\\
\Z^{E-v+1} \oplus \Z_2 \oplus S_{i-2}(G) & \text{otherwise}
 \end{cases} $$
 
 \end{thm}

\begin{proof}
Note that $H_{\A_2}^{i,v(G|P_n)-i}(G|P_n) = H_{\A_2}^{i,(v+n-2)-i}(G|P_n)$. First we consider the case where $i > n-2$. We induct on $n$, the length of the added cycle. For $n=3$, let $e$ be an edge of $P_3$ that is not in $G$.

Observe that $(G|P_3)/e$ is $G$ with a double edge, so $H_{\A_2}((G|P_3)/e) \cong H_{\A_2}(G)$. The graph $G|P_3-e$ is $G$ with a pendant edge. From Lemma \ref{Sum} and  \cite[Proposition 3.4] {HGRb} we obtain the proof for $n=3$: 
\begin{align*}
H_{\A_2}^{i,(v+1)-i}(G|P_3) &\cong H_{\A_2}^{i-1,(v+1)-i}(G|P_3/e) \oplus H_{\A_2}^{i,(v+1)-i}(G|P_3-e) \cong H_{\A_2}^{i-1,(v+1)-i}(G) \oplus H_{\A_2}^{i,v-i}(G)
\end{align*}

The induction step is based on following:
\begin{align*}
H_{\A_2}^{i,(v+n-2)-i}(G|P_n) &\cong H_{\A_2}^{i-1,(v+n-2)-i}(G|P_n/e) \oplus H_{\A_2}^{i,(v+n-2)-i}(G|P_n-e)  
\\
&\cong H_{\A_2}^{i-1,(v+n-2)-i}(G|P_{n-1}) \oplus H_{\A_2}^{i,(v+n-2)-i}(G)\{n-2\}\\
&\cong H_{\A_2}^{i-1,(v+n-3)-(i-1)}(G|P_{n-1}) \oplus H_{\A_2}^{i,v-i}(G)\\
&\cong \left( \bigoplus\limits_{k=0}^{n-3} H_{\A_2}^{(i-1)-k, v-(i-1)+k}(G) \right) \oplus H_{\A_2}^{i,v-i}(G) \cong \bigoplus\limits_{k=0}^{n-2} H_{\A_2}^{i-k, v-i+k}(G).
\end{align*}

For cases where $i \le n-2$, we state the result differently to accommodate the extra $\Z$ in bipartite graphs. We apply Lemma \ref{Sum} a total of $i-1$ times to obtain:

\begin{align*}
H_{\A_2}^{i,(v+n-2)-i}(G|P_n) &\cong  H_{\A_2}^{i-1,(v+n-2)-i}(G|P_{n-1}) \oplus H_{\A_2}^{i,v-i}(G)\\
&\cong H_{\A_2}^{1,(v+n-2)-i}(G|P_{n-(i-1)}) \oplus \bigoplus\limits_{k=0}^{i-2} H_{\A_2}^{i-k, v-i+k}(G)
\end{align*}

Now we compute the first summand in terms of $G$ only, using Proposition \ref{RankDiagonal}:
\begin{align*}
H_{\A_2}^{1,(v+n-2)-i}(G|P_{n-(i-1)}) &= H_{\A_2}^{1,v(G|P_{n-(i-1)})-1}(G|P_{n-(i-1)}) = \begin{cases}
\Z^{E-v+2} & \text{$G|P_{n-(i-1)}$ is bipartite}\\
\Z^{E-v+1} \oplus \Z_2 & \text{otherwise}
\end{cases}
\end{align*}
Since $G|P_{n-(i-1)}$ is bipartite only for $G$ bipartite, $n-i$ odd, we have derived the formulas for the second and third cases.
\end{proof}

The following results are special cases of the previous theorem when graph $G$ is also a cycle. 
\begin{cor} \label{TriEdge}
The rank of $H_{\A_2}^{i,v-i}(P_3|P_n)$ is given by $\rk H_{\A_2}^{i,v-i}(P_3|P_n) = \begin{cases}
1 & 0 \le i \le n-2\\
0 & \text{ otherwise}
\end{cases}$
\end{cor}

\begin{cor} \label{SquareEdge}
The rank of $H_{\A_2}^{i,v-i}(P_4|P_n)$ is given by the following formulas:\\

If $n$ is even, then $\rk H_{\A_2}^{i,v-i}(P_4|P_n) = \begin{cases}
2  & i < n-1 \text{ odd}\\
1  & i < n-1 \text{ even, } i = n-1\\
0  & i \ge n\\
\end{cases}$

If $n$ is odd, then $\rk H_{\A_2}^{i,v-i}(P_4|P_n) = \begin{cases}
2  & 0 < i < n-1 \text{ even}\\
1  & i < n-1 \text{ odd, } i = 0, \ i = n-1\\
0  & i \ge n\\
\end{cases}$
\end{cor}

\begin{cor} \label{PentEdge}
The rank of $H_{\A_2}^{i,v-i}(P_5|P_n)$ is given by $\rk H_{\A_2}^{i,v-i}(P_5|P_n) = \begin{cases}
1  & i = 0,1,n-1,n\\
2  & 1 < i < n-1\\
0 & i > n\\
\end{cases}$
\end{cor}
 Concatenation of sequences $a = (a_1, \ldots, a_k)$ and $b = (b_1, \ldots, b_{\ell})$ is denoted by $a \cdot b=(a_1, \ldots, a_k, b_1, \ldots, b_{\ell})$. 
 Let $a'$ denote the sequence obtained from $a$ by removing its last element; let $\overline{a}$ represent the sequence obtained from a by reversing its order. 
The notation $(a)^p=a \cdot a \cdot \ldots \cdot a$   represents the constant sequence of length $p.$
 We introduce the following notation for special integer sequences, as in \cite{Manion11}:
\begin{align*}
A_p &= (2,1,3,2,4,3,\ldots,p,p-1)\\
C_p &= (1,1,2,2,3,3,\ldots,p,p)
\end{align*}
Torsion in chromatic homology of graphs $G = P_s|P_t$ depends on the parity of $s$ and $t$. Writing $s=2m$ or $s=2m+1$ and $j=2n$ or $j=2n+1$, we denote $M=M(G)=\min\{m,n\}$.

\begin{thm} \label{TwoCycle}
For all graphs of the form $G = P_s|P_t$ ($s, t \ge 3$),  torsion in chromatic homology follows the pattern  $\tor H^{i,v-i}_{\A_2}(G) = \Z_2^{x_i}$ where $x_i$ is the i\it{th} term of the sequences $x=(x_n)_{n\in\mathbb{N}}$ described below: 

\begin{enumerate}
\item[A)] If $G = P_{2n+1}|P_{2m+1}$ then
$x = C_{M-1} \cdot (M)^{2|m-n|+2} \cdot \overline{C}_{M-1}$
for $1 \le i \le 2n+2m-2$.

\item[B)] If $G = P_{2n+1}|P_{2m}$ with $n \le m$, then $x = C_{M-1} \cdot (M)^{2|m-n|+1} \cdot \overline{C}_{M-1}$
for $1 \le i \le 2n+2m-3$.

\item[C)] If $G = P_{2n+1}|P_{2m}$ with $n>m$, then $x = C_{M-1} \cdot M \cdot (M-1,M)^{|m-n|} \cdot \overline{C}_{M-1}$ for $1 \le i \le 2n+2m-3$.

\item[D)] If $G = P_{2n}|P_{2m}$, then $x = A_{M-1} \cdot M \cdot (M-1,M)^{|m-n|} \cdot \overline{C}_{M-1}$
for $1 \le i \le 2n+2m-4$.

\end{enumerate}

\end{thm}

\begin{proof} Based on Theorem \ref{Det} \cite{LS} the torsion pattern follows from the the free part of homology on the $i+j=v$ diagonal.

We prove the result for all $P_s|P_t$ where $s \le t$. This suffices  because graphs $P_s|P_t$ and $P_t|P_s$ are isomorphic. The result holds for $P_3|P_t$, $t \geq 3$ (Corollary \ref{TriEdge}) and $P_4|P_t$, $t \geq 4$ (Corollary \ref{SquareEdge}). It follows that the result holds for $P_s|P_3$ for any $s\geq 3$, and $P_s|P_4$ for any $s\geq 4$,  -- we use this as a base for the induction. 

Next, fix $s \geq 5$ and assume the result holds for $P_s|P_{q}$,  $q<s.$ To show that it holds for $P_s|P_{q},$ $q \geq s$ we consider the following four cases based on the parity of cycle lengths:
 \begin{enumerate}
 \item[A)] Suppose $G = P_s|P_q = P_{2n+1}|P_{2m+1}$, with $M = \min\{m,n\} = n \leq m$. 
 Let $e$ be an edge of $G$ that is contained in $P_{2m+1}$ but not in the other cycle. Then $G/e = P_{2n+1}|P_{2m}$ and $G-e = P_{2n+1}$ with $2m$ pendant edges. By assumption, homology of $G/e$ follows the pattern given in case B): $C_{n-1} \cdot (n)^{2(m-n)+1} \cdot \overline{C}_{n-1}.$ We have $\rk H_{\A_2}^{0,v}(G) = 1$ and $\rk H_{\A_2}^{1,v-1}(G) = 1$ by Proposition \ref{RankDiagonal}. For $i > 1$, Equation \eqref{sumEq} gives:
 \begin{align*}
 &1~2~2~3~3~\ldots~&&(n-2)~&&(n-1)~&&(n-1)~ &&\underbrace{n~\ldots~n}_{2(m-n)+1}~\overline{C}_{n-1}&&(\text{homology of } G/e)\\
 +~&1~0~1~0~1~\ldots~&&1~&&0~&&1 &&~ &&(\text{homology of } G-e)\\
 =~&2~2~3~3~\ldots~~~&&(n-1)~&&(n-1)~&&n~~~ &&\underbrace{n~\ldots~n}_{2(m-n)+1}~\overline{C}_{n-1}  &&(\text{homology of } G)
\end{align*}
The final pattern for $H^{i,v-i}_{\A_2}(G) $ is 
 $$1~1~2~2~3~3~\ldots~~~(n-1)~(n-1)~n~~~ \underbrace{n~\ldots~n}_{2(m-n)+1}~\overline{C}_{n-1}
=~C_{n-1} \cdot (n)^{2(m-n)+2} \cdot \overline{C}_{n-1}$$
 \item[B)] Analogously, case $G = P_s|P_q= P_{2n+1}|P_{2m}$, $M=n \leq m$, builds off of case A). Choosing an edge $e\in G$ that is contained only in $P_{2m}$  means that $G/e = P_{2n+1}|P_{2(m-1)+1}$ and $G-e = P_{2n+1}$ with $2m-1$ pendant edges attached.  
 \item[C)] Notice that $G = P_s|P_q=P_{2n+1}|P_{2m},$ $n>m$ is isomorphic to $G = P_q|P_s = P_{2m}|P_{2n+1}$. In this case $M=m$ and for simplicity of the argument, we choose the edge of the odd cycle which reduces the computation to graph $P_{2m}|P_{2n}$  which belongs to Case D).
 \item[D)] Let $G = P_s|P_q = P_{2n}|P_{2m}$ with $M = \min\{m,n\} = n \leq m$. Select an edge $e$ of $G$ that is contained in $P_{2m}$ but not in the other cycle. Then $G/e = P_{2n}|P_{2(m-1)+1}$ and $G-e = P_{2n}$ with pendant edges attached.  Case C) gives us the homology of $G/e$ if $n < m-1$; if $n = m$ or $n=m-1$, use Case B) instead.
\end{enumerate}
\end{proof}

The proof of the following theorem is omitted, as it closely follows the proof of Theorem \ref{TwoCycle}.

\begin{thm} \label{Patterns2}
For all graphs of the form $G = P_s|^2P_t$ ($s, t \ge 4$),  torsion in chromatic homology follows the pattern  $\tor H^{i,v-i}_{\A_2}(G) = \Z_2^{x_i}$ where $x_i$ is the i\it{th} term of the sequences $x=(x_n)_{n\in\mathbb{N}}$ described below: 

\begin{enumerate}
\item[A)] If $G = P_{2n+1}|^2P_{2m+1}$ with $M = \min\{m,n\}$, then
$x= C_{M-1} \cdot (M)^{2|m-n|+2} \cdot \overline{C}_{M-1}^{\,'}$
for $1 \le i \le 2n+2m-3$.

\item[B)] If $G = P_{2n+1}|^2P_{2m}$ with $n \le m$, then
$x= C_{M-1} \cdot (M)^{2|m-n|+1} \cdot \overline{C}_{M-1}^{\,'}$
for $1 \le i \le 2n+2m-4$.

\item[C)] If $G = P_{2n+1}|^2P_{2m}$ with $n>m$, then
$x = C_{M-1} \cdot M \cdot (M-1,M)^{|m-n|} \cdot \overline{C}_{M-1}^{'}$
for $1 \le i \le 2n+2m-4$.

\item[D)] If $G = P_{2n}|^2P_{2m}$, then
$x = A_{M-1} \cdot M \cdot (M-1,M)^{|m-n|} \cdot \overline{C}_{M-1}^{'}$
for $1 \le i \le 2n+2m-5$.

\end{enumerate}

\end{thm}

\subsection{Vertex gluing of a cycle}

Using ideas outlined in Section \ref{edgeGlue}, we describe the chromatic homology of graphs obtained by gluing a cycle along a vertex of a given graph. These results allow us to give an alternative proof of \cite[Theorem 2]{WW} stating that certain classes of outerplanar graphs are cochromatic.

Corollary \ref{GlueShift}, which follows from Theorem \ref{PolyEdge}, says that gluing a cycle to $G$ at a vertex has the same effect as gluing along a single edge, up to a shift in the $j$-grading.

\begin{cor} \label{GlueShift} 
For any graph $G$ and any $n \ge 3$,
$H_{\A_2}^{i,v-i}(G*P_n) = H_{\A_2}^{i,v-i-1}(G|P_n)$.
\end{cor}

\begin{proof}
The proof is analogous to the proof of Theorem \ref{PolyEdge}, and yields:
$$H_{\A_2}^{i,v(G*P_n)-i}(G*P_n) \cong H_{\A_2}^{i,(v+n-1)-i}(G*P_n) \cong
\begin{cases}
S_{n-2}(G) & i > n-2\\
\Z^{E-v+2} \oplus  S_{i-2}(G) & i \le n-2, n-i \text{ odd, $G$ bipartite}\\
\Z^{E-v+1} \oplus \Z_2 \oplus S_{i-2}(G) & \text{otherwise}
\end{cases}$$
Since $G*P_n$ has one more vertex than $G|P_n$, the formula above implies that $G*P_n$ has the same homology as $G|P_n$ with an upward shift of one $j$-grading.
\end{proof}

The results in this section determine the chromatic homology of graphs constructed iteratively by gluing cycles only along single edges, or along both single edges and vertices. These families of graphs are known as polygon trees and outerplanar graphs, respectively. 

\begin{defn}
A first-order polygon tree is a graph consisting of a single cycle. An $n$th order polygon tree may be constructed by gluing a new cycle along one edge of an $(n-1)$st order polygon tree.
\end{defn}

\begin{defn}
A planar graph is outerplanar if it can be embedded in the plane with all its vertices on the same face.
\end{defn}

\begin{rem}\label{altChar}
The set of outerplanar graphs may be considered a generalization of polygon trees in which cycles are glued along a single edge, glued at a single vertex, or connected by a bridge. An equivalent description is given in \cite[Theorem 4]{Syslo}.
\end{rem}

\begin{thm}\label{Bridge}
Suppose that $G = G_1 * G_2$ and $G_B$ is the graph obtained by expanding the shared vertex into a bridge between $G_1$ and $G_2$. Then $$H_{\A_2}(G_B) = H_{\A_2}(G)\{1\}.$$
\end{thm}

\begin{proof}
Let $J_1$ denote $G_1$ with a pendant edge, $J_2$ denote $G_2$ with a pendant edge. A result for chromatic polynomials (\cite{DKT}, \cite{Zykov}) says that: $P_G(\la) = \frac{P_{G_1}(\la) P_{G_2}(\la)}{\la}$ and $$P_{G_B}(\la) = \frac{P_{J_1}(\la) P_{J_2}(\la)}{\la(\la-1)} = \frac{\Big((\la-1)P_{G_1}(\la)\Big) \Big((\la-1)P_{G_2}(\la)\Big)}{\la(\la-1)} = (\la-1)P_G(\la).$$
Changing variables to $q=\la-1$, we have $P_{G_B}(q) = q P_G(q)$, so $H_{\A_2}(G_B)$ and $H_{\A_2}(G)$ are determined up to a shift of one $q$-grading.
\end{proof}

\begin{defn}
An induced subgraph $H \subseteq G$ is a graph such that $V(H) \subseteq V(G)$ and $E(H)$ contains all edges in $E(G)$ with both endpoints in $H$.
\end{defn}

Note that induced cycles of $G$ are sometimes referred to as ``chordless cycles" or ``pure cycles".

If two polygon-trees have the same collection of induced cycles, they are chromatically equivalent; i.e., they have the same chromatic polynomial (\cite{CL}, \cite{WW}). An analogous result holds for outerplanar graphs with the same collection of induced cycles and the same number of blocks \cite[Theorem 2]{WW}. Corollary \ref{ChiOut} provides another proof of this fact using chromatic homology.
\begin{cor}\label{ChiOut}\cite[Theorem 2]{WW}
The family of all connected outerplanar graphs with $r_k$ induced cycles of length $k$ and $b$ blocks is chromatically equivalent. If $G$ is in this family, and $G^E$ is a polygon tree with the same collection of induced cycles, then $P_G(\la) = (\la-1)^y P_{G^E}(\la)$ where $y$ is the total number of vertex gluings and bridges in $G$.
\end{cor}
\begin{proof} Based on Remark \ref{altChar}, we need to know the effect that gluing two cycles along a single edge, gluing two cycles at a vertex, or connecting two cycles by a bridge has on chromatic graph homology. Theorem \ref{PolyEdge}, Corollary \ref{GlueShift}, and Theorem \ref{Bridge} cover all relevant graph operations. \end{proof}

\begin{cor} \label{SpanOut}
Let $G$ be a connected outerplanar graph with $r_k$ induced cycles of length $k$. Then $$\hspan(H_{\A_2}(G)) = \sum r_k (k-2)+1.$$
\end{cor}
\begin{proof}Follows from \cite[Theorem 2]{WW} and our considerations.\end{proof}

\subsection{Khovanov homology of certain $3$-strand pretzel links}

Note that the graphs $G|P_n$ described in Subsection \ref{edgeGlue}  are instances of multibridge graphs, Figure \ref{Multi},  defined as follows:

\begin{defn}[\cite{DHK}]
The multibridge graph $\ta(a_1, a_2, \ldots, a_k)$ is the graph obtained by connecting two distinct vertices with $k$ internally disjoint paths, each of length $a_k$. In particular, $\ta(a_1, a_2, \ldots, a_k)$ is called a $k$-bridge graph.
\end{defn}

\begin{figure}[h]
  \centering
   \includegraphics[width=0.9\textwidth]{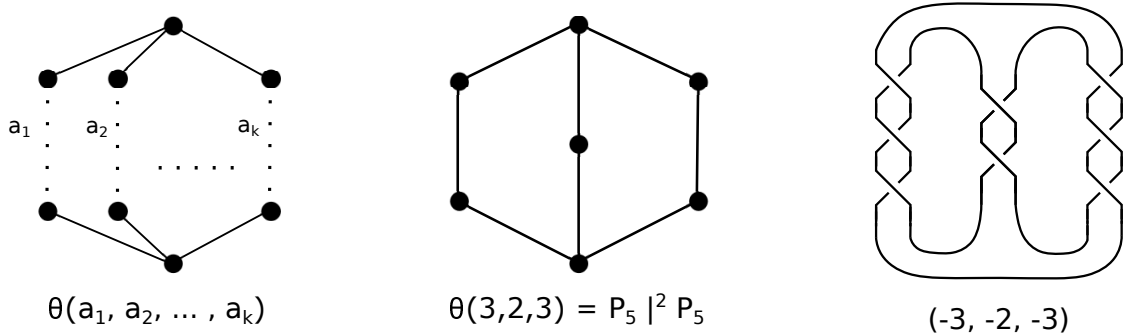}
    \caption{Multibridge graphs (left), multibridge graph $\theta(3,2,3)$ (middle) which can be seen as gluing two pentagons along two edges, that corresponds to the standard diagram of pretzel knot $ (-3, -2, -3)$ (right). }
    \label{Multi}
\end{figure}

Specifically, the 3-bridge graph $\ta(a_1, 1, a_2)$ consists of two cycles $P_{a_1+1}$ and $P_{a_2+1}$ glued along a single edge. We will compute  torsion patterns in chromatic homology  of multibridge graphs of the form $P_n|^kP_m$ when $k=1$ or $k=2$. Note that the graph assigned to the standard diagram of the pretzel knot $K = (-a_1, -a_2, \ldots, -a_n)$, which is $P_{a_1+a_2}|^{a_2}P_{a_2+a_3}|^{a_3} \ldots |^{a_{n-1}}P_{a_{n-1}+a_n},$ where $n \ge 3$ and $a_i \ge 2$ for all $i,$ is precisely the multibridge graph $\ta(a_1, a_2, \ldots, a_k)$. As a corollary, we will be able to partially describe Khovanov homology of alternating 3-strand pretzel knots.

For thin pretzel links, such as those which are alternating or quasi-alternating (\cite{OsSzab2}, \cite{Greene}), torsion is determined by the Jones polynomial and signature via results of Alex Shumakovitch that inspired results in \cite{LS}. Three-strand pretzel links of the form $(p_1, p_2, -q)$ are quasi-alternating if and only if $q > \min\{p_1, p_2\}$ \cite{Greene}.  Rational Khovanov homology of $(p, q, -q)$ is given by a recursive formula on  the parameter $p$ \cite{ Starkston, Qazaqzeh}.  Furthermore, links of the form $(p, q, -q)$ with $q$ odd and $p > q$ are the only non-quasi-alternating pretzels which are homologically thin \cite{Manion13}. The results below describe patterns in Khovanov torsion of alternating links, in terms of the combinatorial properties of the corresponding graph.

\begin{thm} \label{KhPret}
Let $L = (-a_1, \ldots, -a_n)$ be a pretzel link with standard diagram $D$. The homological span of torsion in $Kh(L)$ has the following lower bound: $$\tspan(Kh(L)) \ge \begin{cases}
\displaystyle\min_{1 \le i < j \le n}\{a_i+a_j\}-1 & \text{if $a_i+a_j$ is even for all $i \neq j$}\\
\displaystyle\min_{1 \le i < j \le n}\{a_i+a_j\} & \text{otherwise}
\end{cases}$$
\end{thm}
\begin{proof}
This result is an application of Theorem \ref{KhTor} in the case of alternating pretzel knots. $G_D$ is a graph with 1 block and $\left( \sum_{i=1}^n a_i \right) - n+2$ vertices. By Theorem \ref{Span2}, $\hspan(H_{\A_2}(G_D)) = \left( \sum_{i=1}^n a_i \right) - n+1$. So the last torsion group occurs in grading $i = \left( \sum_{i=1}^n a_i \right) - n$ (\cite{LS}). To prove the result, we need to show that $\left( \sum_{i=1}^n a_i \right) - n$ is greater than or equal to the girth $l$.

The girth of a multibridge graph $G_D = \ta(a_1, a_2, \ldots, a_n)$ is $l = \displaystyle\min_{1 \le i < j \le n}\{a_i+a_j\}.$ Without loss of generality, assume that $\ell = a_{j_1}+a_{j_2}$ and notice that

$$\left( \sum_{i=1}^n a_i \right) - n = \sum_{i=1}^n (a_i-1)  = \left( \sum_{i\neq j_1, j_2} (a_i-1) \right) + (a_{j_1}-1)+(a_{j_2}-1) \ge a_{j_1}+a_{j_2}$$

The last inequality is true when $\sum_{i\neq j_1, j_2} (a_i-1) \ge 2$. This is true for any set of parameters except for $(-2,-2,-2)$ (this case can be verified by direct computation).
\end{proof}

For thin links, torsion in Khovanov homology is determined by the Jones polynomial and signature \cite{Shum2}. No general formula is known for computing this torsion. Using patterns in chromatic homology of multibridge graphs, we can describe a large part of $\Z_2$ torsion in the Khovanov homology of alternating pretzel links.

\begin{exa}[Torsion of pretzel knot and multibridge graph]
The alternating knot with diagram $D=(-3,-2,-3)$, shown in Figure \ref{Multi}, is the mirror of $8_5$ in Rolfsen's table \cite{Rolf, Katl}.  Its corresponding graph is $G = \ta(3,2,3)$, which has girth $5$. In Table \ref{t3} we compare torsion in $Kh(D)$ and $H_{\A_2}(G)$, using boldface to denote matching copies of $\Z_2.$
\begin{table}[h]

\renewcommand{\arraystretch}{1.4}
\begin{center}
  \begin{tabular}{ | p{2cm} | p{1.2cm} | p{1.2cm}  | p{1.2cm}  | p{1.2cm}  | p{1.2cm}  | p{1.2cm}  | p{1.2cm}  | p{1.2cm}  |}
    \hline
    & $p=-5$ & $p=-4$ & $p=-3$ & $p=-2$ & $p=-1$ & $p=0$ & $p=1$ & $p=2$  \\ \hline
    $\tor Kh^{p}(L)$ & $\fakebold \Z_{\mathbf{2}}$ & $\fakebold \Z_{\mathbf{2}}$ & $\fakebold \Z_{\mathbf{2}}^{\mathbf{2}}$ & $\fakebold \Z_{\mathbf{2}}^{\mathbf{2}}$& $\fakebold \Z_{\mathbf{2}}$ & $\Z_2^2$ &  & $\Z_2$ \\ \hline
  \end{tabular}
  \end{center}
  
  \renewcommand{\arraystretch}{1.4}
\begin{center}
  \begin{tabular}{ | p{2cm} | p{1.2cm} | p{1.2cm}  | p{1.2cm}  | p{1.2cm}  | p{1.2cm}  | p{1.2cm}  | p{1.2cm}  | p{1.2cm} |}
    \hline
    & $i=1$ & $i=2$ & $i=3$ & $i=4$ & $i=5$ & $i=6$ & $i=7$ & $i=8$  \\ \hline
    $\tor H^{i}(G)$ & $\fakebold \Z_{\mathbf{2}}$  & $\fakebold \Z_{\mathbf{2}}$  & $\fakebold \Z_{\mathbf{2}}^{\mathbf{2}}$  & $\fakebold \Z_{\mathbf{2}}^{\mathbf{2}}$  & $\fakebold \Z_{\mathbf{2}}$  &  &  & \\ \hline
  \end{tabular}
  \end{center}
  \caption{Torsion in Khovanov homology of pretzel knot $L=(-3,-2,-3)$ and in chromatic homology of the corresponding graph $G = \theta(3,2,3).$ Entries in boldface denote the range where torsion is isomorphic.} 
  \label{t3}
  \end{table}
\end{exa}
The following results are corollaries of the results in Section \ref{AC} that describe patterns in chromatic homology of multibridge graphs.

\begin{cor}\label{PretzelPatterns}
Let $L$ be an alternating 3-strand pretzel link with diagram $D$ such that $D$ has $c_-$ negative crossings and $c_+$ positive crossings and $G_+(D)$ has $v$ vertices. Then $L$ has torsion in Khovanov grading $(i - c_-, v - 2j + c_+ - 2c_-)$ equal to $\Z_2^{x_i}$ where $x_i$ is the i\it{th} term of the sequences $x$ described below: 

\begin{enumerate}
\item[A)] If $D = (-(2n-1), -2, -(2m-1))$ with $m\neq n$, then
$x = C_{M-1} \cdot M \cdot M \cdot M$ for $1 \le i \le 2M+1$, where $M = \min\{m,n\}$.
\item[B)] If $D = (-(2n-1), -2, -(2n-1))$, then
$x = C_{M-1} \cdot M \cdot M \cdot (M-1)$
for $1 \le i \le 2n+1$.
\item[C)] If $D = (-(2n-1), -2, -(2m-2))$ with $n < m$, then $x = C_{M-1} \cdot M \cdot M \cdot M$
for $1 \le i \le 2n+1$.
\item[ D)] If $D = (-(2n-1), -2, -(2m-2))$ with $n \ge m$, then
$x= C_{M-1} \cdot M \cdot (M-1)$
for $1 \le i \le 2m$.
\item[ E)] If $D = (-(2n-2), -2, -(2m-2))$, then
$x = A_{M-1} \cdot M \cdot (M-1)$
for $1 \le i \le 2M$.

\end{enumerate}
\end{cor}

If we take the graph $\ta(a_1, a_2, a_3)$ with a single parameter $a_i = 1$, the corresponding alternating diagram describes a rational 2-bridge link. The Khovanov homology of these links is similar to that of the pretzel links above.
Note that the sequences $C_k$ and $A_k$ in Corollaries \ref{PretzelPatterns} and \ref{RationalPatterns} also appear in the rational homology of non-alternating pretzels \cite{Manion11}.

\begin{cor}\label{RationalPatterns}
Let $L$ be a rational link with Conway notation $-P~Q$ and diagram $D$ with $c_-$ negative crossings and $c_+$ positive crossings. Let $v$ the number of vertices in $G_+(D)$. Then $L$ has torsion in Khovanov grading $(i - c_-, v - 2j + c_+ - 2c_-)$ equal to $\Z_2^{x_i}$ where $k_i$ is the i\it{th} term of the sequences $x$ described below: 

\begin{enumerate}
\item[A)] If $L$ has Conway notation $-(2n+1)~2m+1$ with $m\neq n$, then $x = C_{M-1} \cdot M \cdot M \cdot M$
for $1 \le i \le 2M+1$, where $M = \min\{m,n\}$.

\item[B)] If $L$ has Conway notation $-(2n+1)~2n+1$, then
$x = C_{M-1} \cdot M \cdot M \cdot (M-1)$
for $1 \le i \le 2n+1$.

\item[C)] If $L$ has Conway notation $-(2n+1)~2m$  with $n < m$, then
$x = C_{M-1} \cdot M \cdot M \cdot M$
for $1 \le i \le 2n+1$.

\item[D)] If $L$ has Conway notation $-(2n+1)~2m$ with $n \ge m$, then
$x = C_{M-1} \cdot M \cdot (M-1)$
for $1 \le i \le 2m$.

\item[E)] If $L$ has Conway notation $-2n~2m$, then
$x= A_{M-1} \cdot M \cdot (M-1)$
for $1 \le i \le 2M$.
\end{enumerate}
\end{cor}


 \section{The torsion in the 4th and 4th-ultimate Khovanov homology groups and the corresponding Jones coefficients}\label{Jones}

Chromatic graph homology over algebra $\A_2$ has proven to be useful for providing explicit formulae for the first few extremal homological gradings Khovanov homology subject to combinatorial conditions on the Kauffman state of a link diagram. 
The torsion groups in chromatic homology in degrees ${i, v-i}$ for $i=1,2,3,$ are computed explicitly in \cite{AP,PPS,PS} and used to get the following gradings in Khovanov homology when the isomorphism theorem holds.

\begin{prop}[\cite{PPS,PS}]
Let $D$ be a diagram of $L$ with $c_+$ positive crossings and $c_-$ negative crossings.

\noindent \begin{align*}
Kh^{-c_{-},-v+c_{+}-2c_{-}}(L) &= \Z\\
Kh^{-c_{-},-v+2+c_{+}-2c_{-}}(L) &=  \begin{cases}
\Z & G \text{ bipartite}\\
0 & \text{ otherwise}
 \end{cases} \\
Kh^{1-c_{-},-v+2+c_{+}-2c_{-}}(L)  &= \begin{cases}
\Z^{p_1} & G \text{ bipartite}\\
\Z^{p_1-1} \oplus \Z_2 & \text{ otherwise}
\end{cases}\\
Kh^{2-c_{-},-v+4+c_{+}-2c_{-}}(L)  &=  \begin{cases}
\Z^{\binom{p_1}{2}} \oplus \Z_2^{p_1} & G \text{ bipartite}\\
\Z^{\binom{p_1}{2}-t_3+1} \oplus \Z_2^{p_1-1}& \text{ otherwise}
 \end{cases}
 \end{align*}
\end{prop}

 We use recent results from \cite{LS} and the formulas for coefficients of $P_G(\la)$ given in \cite{Farrell}, to calculate the torsion in $(4, v-4)$ grading as well. Note that \cite[Theorem 2]{Farrell} can be used to compute torsion in degree $(5, v-5)$ of chromatic homology.  We have omitted this formula due to its complexity -- the computation would involve eight possible subgraphs of $G.$
 
\begin{thm}\label{Farrell}\cite{Farrell}
Let $G$ be a graph with $t_3$ triangles, $t_4$ induced 4-cycles, and $k_4$ complete graphs of order 4. The first four coefficients of the chromatic polynomial
$$P_G(\la) = c_v\la^v + c_{v-1}\la^{v-1} + c_{v-2} \la^{v-2} + c_{v-3} \la^{v-3}+ \ldots$$
are given by the following formulas: $c_v = 1$, $c_{v-1} = -E$, $c_{v-2} = \binom{E}{2}-t_3$, and $c_{v-3} = -\binom{E}{3}+(E-2)t_3+t_4-k_4$. 

\end{thm}

\begin{thm}\label{4thKh}
\begin{align*}
\rk H_{\A_2}^{3,v-3}(G) &=  \begin{cases}
p_1 + \binom{p_1+1}{3}-t_4 & G \text{ bipartite}\\
p_1 + \binom{p_1+1}{3}-t_3(p_1-1)-t_4+2k_4-1 & \text{ otherwise}
 \end{cases}\\
\tor H_{\A_2}^{4,v-4}(G) &=  \begin{cases}
\Z_2^{p_1 + \binom{p_1+1}{3}-t_4} & G \text{ bipartite}\\
\Z_2^{p_1 + \binom{p_1+1}{3}-t_3(p_1-1)-t_4+2k_4-1} & \text{ otherwise}
 \end{cases}
\end{align*}
 
 \end{thm}
 
\begin{proof}

Let the chromatic polynomial of $G$ have coefficients labeled as follows: $$P_G(\la) = \la^v + c_{v-1}\la^{v-1} + \ldots + c_2 \la^2 + c_1 \la$$
The change of variable $\la = q+1$ gives
\begin{align*}
P_G(q) &= (q+1)^v + c_{v-1}(q+1)^{v-1} + \ldots + c_2 (q+1)^2 + c_1 (q+1)\\
&= q^v + a_{v-1}q^{v-1} + \ldots + a_2 q^2 + a_1q + a_0.
\end{align*}
Since chromatic homology is supported on only two diagonals,  $a_{v-3} = \rk H_{\A_2}^{2,v-3}(G) - \rk H_{\A_2}^{3,v-3}(G)$. By \cite[Cor. 4.2]{CCR}, $\rk H_{\A_2}^{2,v-3}(G) = \rk H_{\A_2}^{1,v-1}(G)$. The rank of $H_{\A_2}^{1,v-1}(G)$ is known (see Proposition \ref{RankDiagonal}), so
$$\rk H_{\A_2}^{3,v-3}(G) = \rk H_{\A_2}^{2,v-3}(G) - a_{v-3} = \rk H_{\A_2}^{1,v-1}(G) - a_{v-3} = \begin{cases}
p_1 - a_{v-3} & G \text{ bipartite}\\
p_1-1 - a_{v-3} & \text{ otherwise}
 \end{cases}$$
Using formulas in Theorem \ref{Farrell}, we compute $a_{v-3}$.
\begin{align*}
a_{v-3} &= \binom{v}{v-3} + c_{v-1}\binom{v-1}{v-3} + c_{v-2}\binom{v-2}{v-3} + c_{v-3}\\
&= \binom{v}{v-3} -E \binom{v-1}{2} + \left( \binom{E}{2}-t_3\right)(v-2) - \binom{E}{3} + (E-2)t_3+t_4-2k_4\\
&= -\frac{1}{6}(E-v)(1+E-v)(2+E-v)+t_3(E-v)+t_4-2k_4\\
&= -\binom{p_1+1}{3}+t_3(p_1-1)+t_4-2k_4
\end{align*}
Note that $t_3 = k_4 = 0$ if $G$ is bipartite.
\end{proof}

For a reduced alternating diagram $D$, the first three coefficients of the Jones polynomial may be stated in terms of the all-$A$-state graph $A(D)$ \cite{DasLin} which is equivalent to our all-positive graph $G_{+}(D)$. If $G_+(D)$ has girth greater than or equal to 4, the formula for $\rk H_{\A_2}^{3,v-3}(G)$ in the last proof gives us the fourth coefficient of the Jones polynomial.

\begin{thm}\label{rankKh}
Let $D$ be a diagram of $L$ with $c_+$ positive crossings and $c_-$ negative crossings, whose corresponding graph $G_+(D)$ has girth at least 4.
\begin{align*}
\rk Kh^{3-c_-,-v+6+c_+-2c_-}(L) &\cong \begin{cases}
p_1 + \binom{p_1+1}{3}-t_4 & G_{+}(D) \text{ bipartite}\\
p_1 + \binom{p_1+1}{3}-t_3(p_1-1)-t_4+2k_4-1 & \text{ otherwise}
\end{cases}\\
\tor Kh^{4-c_-,-v+8+c_+-2c_-}(L) &\cong \begin{cases}
\Z_2^{p_1 + \binom{p_1+1}{3}-t_4} & G_{+}(D) \text{ bipartite}\\
\Z_2^{p_1 + \binom{p_1+1}{3}-t_3(p_1-1)-t_4+2k_4-1} & \text{ otherwise}
\end{cases}
\end{align*}
\end{thm}

\begin{thm}
Let $D$ be a diagram of a link $L$ such that $Kh(L)$ is homologically thin. Let the Jones polynomial of $L$ be written as $$J_L(q) = aq^{C} + bq^{C+2} + cq^{C+4} + dq^{C+6} + \ldots$$ with positive first coefficient. If $\ell(G_+(D)) \ge 4$, then the fourth ultimate coefficient of the Jones polynomial is $$d = -\binom{p_1+2}{3}+t_4$$ where $p_1$ is the cyclomatic number of $G_+(D)$ and $t_4$ is the number of induced 4-cycles.
\end{thm}
\begin{proof}
We write the unnormalized Jones polynomial with coefficients $\al, \be, \gamma, \delta$:
\begin{align*}
\hat{J}_L(q) &= (q+q^{-1})J_L(q)\\
 &= (aq^{C-1}+aq^{C+1})+(bq^{C+1}+bq^{C+3})+ (cq^{C+3}+cq^{C+5}) +(dq^{C+5}+dq^{C+7}) + \ldots \\
&= aq^{C-1} + (a+b)q^{C+1}+ (b+c)q^{C+3}+ (c+d)q^{C+5}+ \ldots\\
&= \al q^{C-1} + \beta q^{C+1}+ \gamma q^{C+3}+ \delta q^{C+5}+ \ldots
\end{align*}
Since $Kh(L)$ lies only on two diagonals, the isomorphism of Theorem \ref{Correspondence} implies that the first four coefficients $\al, \be, \gamma, \delta$ of $\hat{J}_L$ are equal to the first four coefficients of the chromatic polynomial $P_{G_{+}(D)}$. By Proposition \ref{RankDiagonal} and the isomorphism in \cite[Cor. 4.2]{CCR}, we have the following coefficients for $P_{G_{+}(D)}$. Note that since $G_{+}(D)$ has girth greater than 3, $t_3$ and $k_4$ are zero in the formulas from Proposition \ref{RankDiagonal}.
\begin{align*}
\al &= 1\\
\be &= \rk H_{\A_2}^{0, v-1} - \rk H_{\A_2}^{1, v-1} = 1-p_1\\
\gamma &= -\rk H_{\A_2}^{1, v-2} + \rk H_{\A_2}^{2, v-2} = \binom{p_1}{2}-t_3 = \binom{p_1}{2}\\
\delta &= \rk H_{\A_2}^{2, v-3} - \rk H_{\A_2}^{3, v-3}= -\binom{p_1+1}{3}+t_3(p_1-1)+t_4-2k_4 = -\binom{p_1+1}{3}+t_4
\end{align*}
The coefficients in the normalized version of the Jones polynomial are obtained as follows: 
$a =\al = 1$, $b = \be - a = -p_1$, $c = \gamma-b = \binom{p_1+1}{2}$, $d = \delta-c = -\binom{p_1+2}{3}+t_4.$
\end{proof}


\section{Existence of gaps in Khovanov and chromatic homology} \label{Gaps}

We prove several results concerning gaps in torsion for $H_{\A_2}(G)$ and their analogues for Khovanov homology of corresponding diagrams via Theorem \ref{Correspondence}. 
\begin{defn}
Let $H$ be either Khovanov or chromatic homology.
 A homological torsion gap of  $H$ of length $g$ exists if there exists $i$ in the span of homology such that $ H^{i-1}(G)$  and $H^{i+g}(G)$ has torsion, but $H^k(G)$ does not for $i \leq k <i+g. $
\end{defn}
Notice that the quantum torsion gap can be defined analogously and that for  chromatic homology over $\A_2$ and thin Khovanov homology, a homological gap in torsion is necessarily a quantum torsion gap, since torsion exists only on one diagonal. 

Since there is a single $\Z_2$ in $H_{\A_2}^{1,v-1}(G)$ if $G$ has an odd cycle, and no torsion if $G$ is bipartite, the following definition involves homology in degrees two and higher. 

\begin{defn}
Torsion of chromatic homology $H_{\A_2}(G)$  over algebra $\A_2$  is said to be dense if there is at least one $\Z_2$ in every $i$-grading from $i=2$ to $i=v-b-1$, i.e. if there are no homological torsion gaps.
\end{defn}

\begin{thm} \label{2dense}
Chromatic homology  $H_{\A_2}(P_m|P_n)$ of two polygons $P_n, P_m$  for $m,n \ge 3$ glued along an edge has dense torsion.
\end{thm}
\begin{proof} 
Having dense torsion means that  $H_{\A_2}^{i,(m+n-2)-i}(P_m|P_n)$ contains torsion for every $2 \le i \le m+n-4$. Based on Theorem \ref{PolyEdge} we consider the following cases. 

For $n-2 < i \le m+n-4$, we use the formula $H_{\A_2}^{i,(m+n-2)-i}(P_m|P_n) \cong \bigoplus\limits_{k=0}^{n-2} H_{\A_2}^{i-k, m-i+k}(P_m).$
Since $n \geq 3$, the sum in the formula must include at least the $k=n-2$ and $k=n-3$ terms. That means we are looking into one of the  $H_{\A_2}^{i^*, m-i^*}(P_m)$ with $i^*=i-k$, $2 \leq k \leq m-2$ which must contain a copy of $\Z_2$ by Corollary \ref{PolygonLemma}. 
Suppose $i \le n-2$, $n-i$ is odd, and $m$ is even. The corresponding formula from Theorem \ref{PolyEdge} 
$H_{\A_2}^{i,(m+n-2)-i}(P_m|P_n) \cong \Z^{E-v+2} \oplus  \bigoplus\limits_{k=0}^{i-2} H_{\A_2}^{i-k, m-i+k}(P_m)$
 contains $H_{\A_2}^{2, m-2}(P_m)=\Z_2$  for $k=i-2.$ 

For all other $i \le n-2$ not covered by case 2, there is $\Z_2$ torsion by the third formula in Theorem \ref{PolyEdge}.
\end{proof}

\begin{cor} \label{PolyTreeDense}
If $G$ is a polygon-tree, then $H_{\A_2}(G)$ has dense torsion.
\end{cor}

\begin{proof}
We induct on the number of cycles in $G$ with Theorem \ref{2dense} as our base case for a graph $G_2$ with only two cycles. Assume the result holds for all polygon-trees with $p-1$ cycles, where $p \ge 3$. For the induction step we show that if $G_{p-1}$ is one such graph, then the result also holds for any $G_p = G_{p-1}|P_n$, $n \ge 3$.

If $v$ denotes  the number of vertices in $G_{p-1}$ then $G_p$ has $v+n-2$ vertices and the gradings of interest are $2 \le i \le v+n-4$. According to Theorem \ref{PolyEdge} there are three cases. 

 For $n- 1 \leq i \le v+n-4$, Theorem \ref{PolyEdge} yields $
H_{\A_2}^{i,(v+n-2)-i}(G_p) \cong \bigoplus\limits_{k=0}^{n-2} H_{\A_2}^{i-k, v-i+k}(G_{p-1}) $. As in the proof of Theorem \ref{2dense} we show that there exists a term that contributes at least one copy of $\Z_2.$ By induction hypothesis, if $i = n-1$, then the term with $k=n-3$ contains torsion, otherwise the same is true for $k=n-2.$ Similar arguments apply in the remaining two cases. 
\end{proof}

The following two Theorems are based on Corollary \ref{GlueShift} and Corollary \ref{PolyTreeDense}, respectively. 

\begin{thm}
If $G$ is a connected outerplanar graph, $H_{\A_2}(G)$ has dense torsion.
\end{thm}
\begin{thm}
If $H_{\A_2}(G)$ has dense torsion, the same is true of $H_{\A_2}(G|P_n)$ and $H_{\A_2}(G*P_n)$ for $n \ge 3$.
\end{thm}

As a corollary we get the existence of $\Z_2$ torsion in Khovanov homology of some link provided that it can be associated a graph with certain properties. 

\begin{cor}
Let $D$ be a diagram of link $L$ such that $G = G_+(D)$ is a polygon-tree or bridge-free outerplanar graph with $v,b,\ell$ are the number of vertices, number of blocks, and girth of $G$. Then there is $\Z_2$ torsion in Khovanov homology $Kh^{p, p-v+c_+-c_-}(L)$ of the corresponding link for $2-c_- \le p \le \min\{\ell, v-b-1\}- c_{-}.$ \end{cor} 
\begin{thm}
Let $L$ be an alternating 3-strand pretzel link with a diagram $D$ given by Conway symbol  $-2, -2, -(n-2)$ where $n \ge 4$. Then there is a homological gap in torsion of its Khovanov homology $Kh(L)$.
\end{thm}
\begin{proof}
Note that diagram $D$ corresponds to a multibridge graph $G_+(D)= \ta(2, n-2, 2)$ which  has 1 block, $n+1$ vertices and girth $n$. By Theorem \ref{Span2}, $H_{\A_2}(G_+(D))$ has no homology in grading $i=n$, so $Kh^{p,q}(L)$ has no torsion in the corresponding Khovanov grading $p= n-c_-$.
\end{proof}


\section{Chromatic homology over $\A_m$}\label{AlgebraAm}

 In this section we provide generalizations of some of the results and patterns observed in chromatic homology over $\A_2$ to the algebra $\A_m=\Z[x]/(x^m=0)$, focusing  on $m=3.$  We show that some properties which are constant over $\A_2$, such as width, become dependent both on the choice of algebra $\A_m$ for $m>2$, and and the choice of graph. These preliminary results indicate that chromatic homology may have richer algebraic structure over other algebras and may be better at distinguishing graphs.

\subsection{Width of chromatic homology over $\A_m$}

Computations indicate that the homological span of chromatic homology is invariant under the choice of algebra $\A_m.$
\begin{conj}\label{SpanMc}
The homological span of chromatic homology over algebra $\A_m$ of any graph $G$ with $v$ vertices and $b$ blocks is equal to  $\hspan(H_{\A_m}(G)) = v-b.$\end{conj}

At the moment, we can only show that we have a lower bound on width following the reasoning in Theorem \ref{Span2} and basic results from \cite{HPR}: 
\begin{thm}\label{SpanM}
Homological span of chromatic homology over any algebra $\A_m$ depends only on the number of vertices $v$ and blocks $b$ of a graph G: $\hspan(H_{\A_m}(G)) \ge v-b.$
\end{thm}

It is interesting that, unlike the case of $\A_2$ where width is equal to two,  the width of the chromatic homology increases with $m$ and depends on the number of vertices of the graph.

\begin{thm} \label{WidthM} For any graph $G$ the width of $H_{\A_m}(G)$ is equal to $\width(H_{\A_m}(G)) = (m-2)v+2$.

\end{thm}

\begin{proof}
In case that $G$ is a tree note that $|E(G)| = v-1$. Next  note that $H_{\A_m}^0(G) = \A_m \otimes \A_m'^{\otimes^{v-1}}$ where $\A_m'$ is the submodule of $\A_m$ such that $\A_m=\Z_{\bf 1} \oplus \A'$ with ${\bf 1}$ the identity of $\A_m$ \cite[Proposition 3.4, Example 4.3] {HGRb}.  
Therefore the  highest non-zero homology group is $H_{\A_m}^{0,(m-1)v}(G) = \Z$, on the diagonal $i+j=(m-1)v$.  The lowest non-zero group in $\A_m$ is located on the diagonal $i+j=v-1$, so $\width(H_{\A_m}(G)) = (m-1)v-(v-1)+1 = (m-2)v+2$.

If  $G$ is not a tree we still have $H_{\A_m}^{0,(m-1)v}(G) = \Z$. It remains to show that there exists a non-trivial entry on  $i+j = v-1$ diagonal; i.e., that there exists $j > 0$ such that $H_{\A_m}^{v-1-j,j}(G) \neq 0$. Arguments in the proof of Theorem \ref{Span2} generalize to $\A_m$ to show that $H_{\A_m}^{v-b-1, b}(G)$ is non-trivial, which is precisely the group we needed.
\end{proof}

Considering homological span of torsion is somewhat more involved. Note that Hochschild homology implies the following about the span of torsion for cycle graphs:

\begin{prop}
For $m>2$, $H_{\A_m}(P_n)$ has one $\Z_m$ torsion group on each of  $\left\lceil \frac{n}{2}-1 \right\rceil$  diagonals.
\end{prop}

\begin{prop} \label{htPn}
The torsion width of chromatic homology of a cycle is given by
$$\twidth(H_{\A_m}(P_n)) = \begin{cases}
\frac{mn}{2}-2m-n+5, & \text{$n$ even}\\
\frac{mn}{2}-\frac{3}{2}m-n+4, & \text{$n$ odd}\\
\end{cases}$$
\end{prop}

We conjecture that the width of torsion over $\A_3$ of any graph depends only on the number of vertices and the girth of the graph.

\begin{conj} \label{TWidth3}
Let $G$ be a simple, connected graph with $v$ vertices and girth $\ell$, with $\ell = 2k$ or $\ell = 2k-1$ depending on parity. Then $\twidth(H_{\A_3}(G)) = \twidth(H_{\A_3}(P_\ell)) + v - \ell = (k-1)+v - \ell = \begin{cases}
v-k-1, & \text{$\ell$ even} \\
v-k, & \text{$\ell$ odd} \\
\end{cases}$
\end{conj}

\subsection{$H^{i_{max}}(G)$ tail of homology}


The fact that chromatic homology $H_{\A_2}(G)$ is supported on two diagonals, has the knight move structure \cite{CCR}, contains no torsion other than $\Z_2$ and is completely determined by the chromatic polynomial \cite{LS} enables us to describe the homology in the maximal homological grading $i_{max}.$
$H^{i_{max}}_{\A_2}(G)$ contains a free group on the lowest diagonal, and since $H_{\A_2}^{v-b-1,b}(G)=\Z^k$ is the only group in $j_{min}$,  $k$ is equal to the absolute value of the coefficient on the lowest degree term in $P_G(1+q)$.
The only other non-trivial group in maximal homological grading $i_{max}=v-b-1$ is $H_{\A_2}^{v-b-1,b+1}(G)$ and it contains a copy of $\Z_2$ for every copy of $\Z$ in $H_{\A_2}^{v-b-1,b}(G).$ 
In the rest of the Section we will refer to $H_{\A_2}^{i_{max}}(G)= H_{\A_2}^{v-b-1}(G)$ as the tail of chromatic homology of  $G$ and denote it as $Tl_{\A_2}(G)$. 
Notice that the tail of a cycle $P_n$ is $Tl_{\A_2}(P_n)=\renewcommand{\arraystretch}{1.5}
\begin{tabular}{|c|}\hline
 $\Z_2$ \\ \hline
$\Z$ \\ \hline
\end{tabular}.$

The tail of any graph consists of some number of copies of $Tl_2:=Tl_{\A_2}(P_n).$  The rest of the section contains explicit computations of the tail of chromatic homology based on knowing the lowest coefficient of $P_G(1+q)$. 

\begin{thm} \label{TailOut}
If $G$ is a connected outerplanar graph, then $Tl_{\A_2}(G)=Tl_2.$
\end{thm}

\begin{proof}
If $G$ has $r_k$ $k$-gons and $b$ blocks, then $P_G(\la) = (-1)^{n} \la (\la-1)^b \prod_{k \ge 3} (1+(1-\la)+(1-\la)^2+ \ldots + (1-\la)^{k-2})^{r_k}$ where $n = \displaystyle \sum_{k \ge 3} r_k (k-2)$ (\cite[Theorem 2]{WW}).
Under the variable change $\la = 1+q$, the chromatic polynomial of $G$ becomes the $q$-polynomial $$(-1)^{n} (1+q) q^b \prod_{k \ge 3} (1+(-q)+(-q)^2+ \ldots + (-q)^{k-2})^{r_k}$$ The lowest degree term in this polynomial has coefficient $\pm 1$ so we get only one copy of $Tl_2$ in the tail of $G$.
\end{proof}

A \textit{chord} is an edge that joins two vertices of $P_n$ but is not itself an edge of $P_n$. A \textit{chordal graph} is one in which every cycle of length 4 or higher has a chord. In other words, chordal graphs contain no induced cycles of length greater than $3.$

\begin{thm}\label{TailChord}
 If $G$ is a chordal graph, $Tl_{\A_2}(G)$ is the direct sum of $2^{s_3} 3^{s_4} \cdots (k-1)^{s_k}$ copies of $Tl_2$, where $s_k$ is the exponent of $(\la-k)$ in $P_G(\la)$.\end{thm}
\begin{proof}
If $G$ is a chordal graph with $v$ vertices, then $P_G(\la) = \la^{s_0}(\la-1)^{s_1}(\la-2)^{s_2} \cdots (\la - k)^{s_k}$ with $s_i \geq 0, \, \forall i$ such that $\sum_{i=0}^k s_i = v$ (\cite{DKT}). Next $P_G(1+q)=(1+ q)^{s_0}(q)^{s_1}(-1+ q)^{s_2} \cdots (-(k-1)+ q)^{s_k}$ whose lowest degree term is $(-1)^S 2^{s_3} 3^{s_4} \cdots (k-1)^{s_k} q^{s_1}$, where $S = \sum_{i=1}^k s_i$. The absolute value of the coefficient of the lowest degree term is $2^{s_3} 3^{s_4} \cdots (k-1)^{s_k}$.
\end{proof}

\begin{cor}\label{TailKn} Let $K_n$ denote the complete graph on $n$ vertices and $W_n$ the wheel graph. Then $Tl_{\A_2}(K_n) \cong Tl_2^{\oplus (n-2)!}$, and  $Tl_{\A_2}(W_n) = Tl_2^{\oplus (n-2)}$.
\end{cor}
\begin{proof}
We use the formulas $P_{K_n}(q) = (q+1)q(q-1) \cdots (q-(n-2))$ \cite[Example 1.2.2]{DKT} and $P_{W_n}(q) = (q+1)\left((q-1)^{n-1} + (-1)^{n-1}(q-1) \right)$ \cite[Cor. 1.5.1]{DKT}. For the second formula, note that the constant term of the second factor is always zero, while the $q$ term will be $((n-1)-1)q = (n-2)q$ if $n$ is even, and $(-(n-1)+1)q = -(n-2)q$ if $n$ is even.
\end{proof}

\begin{conj}
Let $W_n^{in}$ be the graph obtained from $W_n$ by removing an edge that connects the central vertex to one of the outer vertices. Then the tail of $Tl_{\A_2}(W_n^{in}) =Tl_2^{\oplus (n-3)}$.
\end{conj}

It is natural to ask if this phenomenon extends to chromatic homology over other algebras. The Hochschild homology of algebra $\A_m$ gives us that the tail of $H_{\A_m}(P_n)$ denoted by $Tl_m:=H_{\A_m}^{n-2}(P_n)$ has the same ``shape" as over $\A_2$:  $m-1$ copies of $\Z$ with a $\Z_m$ in the highest quantum grading. 

We conjecture that the tail of any graph $Tl_{\A_m}(G)$  consist of some number of copies of the tail of a cycle, but proving this statement would require structure theorems such as those existing in the $\A_2$ case. For example, computations for small values of $n,m$ hint that Corollary \ref{TailKn} extends to other $\A_m$ in the case of complete graphs.

\begin{conj}
The tail of the complete graph $K_n$ in chromatic homology over $\A_m$ consist of $(n-2)!$ copies of the tail of  $P_n,$ i.e. $Tl_{\A_m}(K_n) = Tl_m^{\oplus (n-2)!}$ for $m >3.$\end{conj}

\subsection{Relative strengths of chromatic homology and graph polynomials}


Although the chromatic homology over $\A_2$ is completely determined by the chromatic polynomial, there are examples of cochromatic graphs distinguished by the chromatic homology over $\A_3$, see \cite{PPS}.  The difference appearing in \cite[Example 6.4]{PPS} may be explained in terms of edge gluing. In this section we list  several examples of cochromatic graphs distinguished by chromatic homology over $\A_3$, none of which differ by only an edge product described in Section \ref{AC}.

\begin{exa}
\begin{figure}[h]
  \centering
    \includegraphics[width=0.6\textwidth]{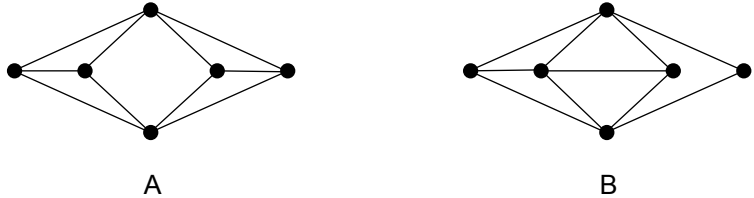}
   \caption{An example of cochromatic graphs from \cite{BM}.}
   \label{A3example1}
\end{figure}
The graphs in Figure \ref{A3example1} appear in \cite[Exercise 8.4.1]{BM} and share the following chromatic polynomial:
$\la^6-10\la^5+41\la^4-84\la^3+84\la^2-32\la.$ However,  $H^{1,9}_{\A_3}(A) = \Z^7 \oplus \Z_3^3$, which differs from $H^{1,9}_{\A_3}(B) = \Z^8 \oplus \Z_3^3$.
\end{exa}

\begin{exa} Cochromatic graphs in Figure \ref{A3example2} from \cite{CW} and have the following chromatic polynomial: $\la^6-10\la^5+40\la^4-80\la^3+79\la^2-30\la.$ 
Their first chromatic cohomology differ in quantum degree $9$: $H^{1,9}_{\A_3}(A) = \Z^5 \oplus \Z_2 \oplus \Z_3^5$, $H^{1,9}_{\A_3}(B) =  \Z^6 \oplus \Z_2 \oplus \Z_3^4$.
\begin{figure}[h]
  \centering
   \includegraphics[width=0.6\textwidth]{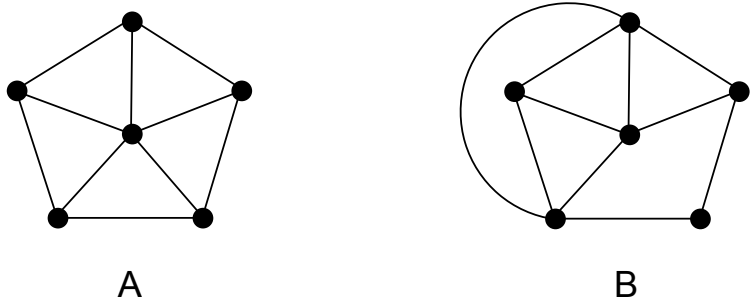}
      \caption{First example of cochromatic graphs from \cite{CW}.}
         \label{A3example2}
\end{figure}
\end{exa}

\begin{exa}
The graphs in Figure \ref{A3example3}, also found in \cite{CW}, share the following chromatic polynomial:
$$\la^7-11\la^6+51\la^5-128\la^4+184\la^3-143\la^2+46\la$$
but $H^{1,11}_{\A_3}(A) = \Z^4 \oplus \Z_3^4$, which differs from $H^{1,11}_{\A_3}(B) =  \Z^5 \oplus \Z_3^3$.
\begin{figure}[h]
  \centering
      \includegraphics[width=0.6\textwidth]{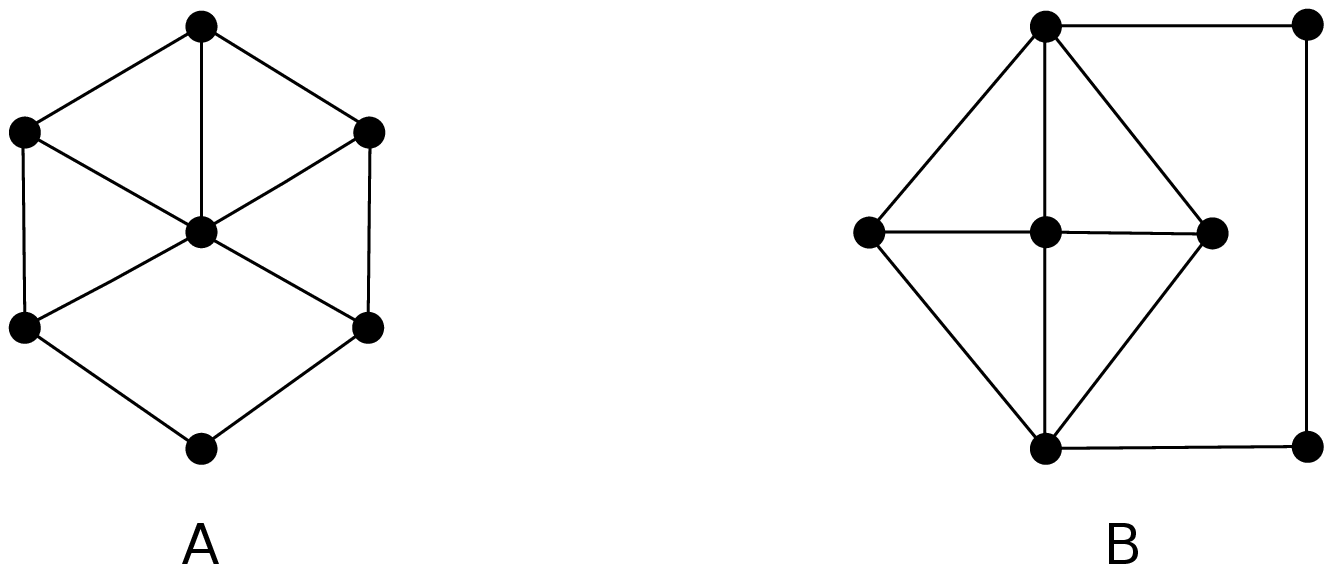}
     \caption{Second example of cochromatic graphs from \cite{CW}.}
        \label{A3example3}
\end{figure}
\end{exa}

\begin{figure}[h]
  \centering
    \includegraphics[width=0.6\textwidth]{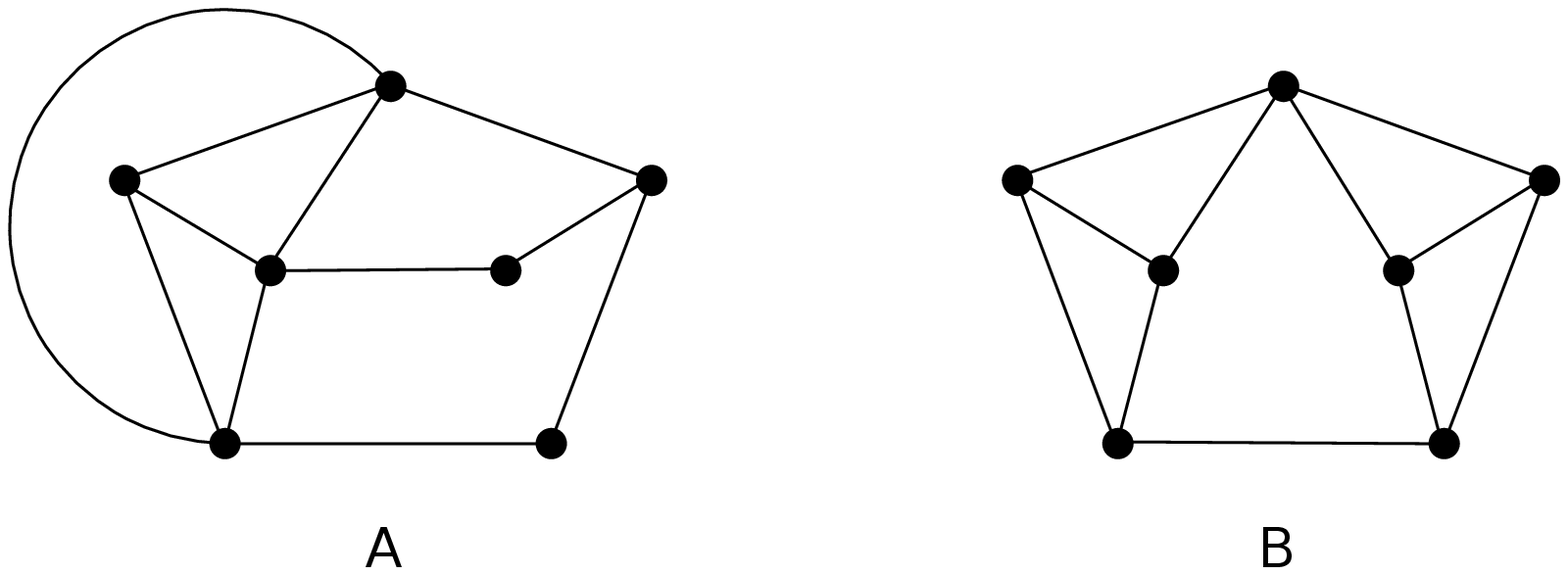}
   \caption{An example of cochromatic graphs from \cite{KG}.}
      \label{A3example4}
\end{figure}
\begin{exa}
The graphs in Figure \ref{A3example4} appear in \cite{KG} and share the following chromatic polynomial:
$$\la^7-11\la^6+51\la^5-129\la^4+188\la^3-148\la^2+48\la$$ 
but $H^{1,11}_{\A_3}(A) = \Z^4 \oplus \Z_2 \oplus \Z_3^3$, which differs from $H^{1,11}_{\A_3}(B) =  \Z^2 \oplus \Z_3^4$.
\end{exa}

\begin{figure}[h]
  \centering
     \includegraphics[width=0.6\textwidth]{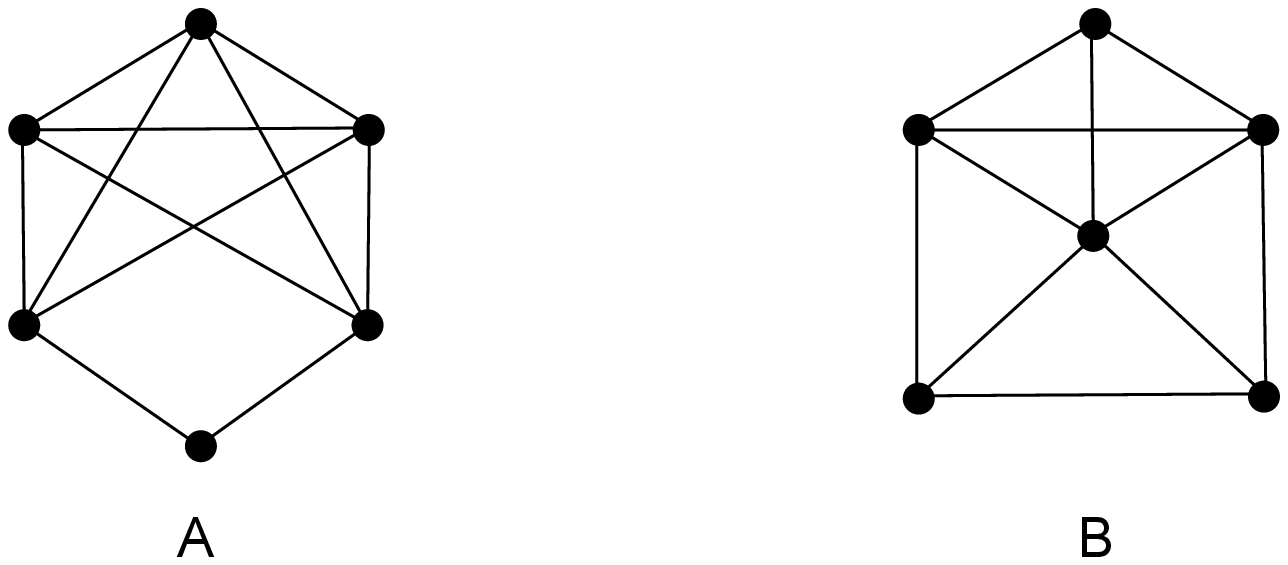}
   \caption{An example of cochromatic graphs from \cite{DKT}.}
    \label{A3example5}
\end{figure}
\begin{exa}
The graphs in Figure \ref{A3example5} appeared in \cite{DKT} (attributed to unpublished work by Chee and Royle). They share the following chromatic polynomial:
$$\la^6-11\la^5+48\la^4-103\la^3+107\la^2-42\la$$
but  $H^{1,9}_{\A_3}(A) \cong \Z^{10} \oplus \Z_2 \oplus \Z_3^4$, which differs from  $H^{1,9}_{\A_3}(B) \cong  \Z^9 \oplus \Z_2 \oplus \Z_3^5$.
Note that  \cite[Example 3.2.3] {DKT} has additional examples of cochromatic graphs for which the computation of $H_{\A_3}$ exceeds our current resources.
\end{exa}

Two connected graphs $G_1$ and $G_2$ are \textit{2-isomorphic} if they differ only by a Whitney twist or a single vertex attachment \cite[Thm. 5.3.1]{Oxley}. Equivalently, $G_1$ and $G_2$ have the same graphic matroid (also known as the cycle matroid), whose independent sets are acyclic sets of edges. The Tutte polynomial of a graph is determined by its graphic matroid (see \cite{Mier}, e.g.).

It turns out that  $H^{1, 2v-3}_{\A_3}(G)$  is preserved under 2-isomorphisms \cite[Theorem 6.2]{PPS}.  The following example demonstrates that the next quantum grading, $j=2v-4$, can distinguish graphs with the same Tutte polynomial or the same 2-isomorphism class.

\begin{exa}[Chromatic homology vs. the Tutte polynomial] 
The graphs in Figure \ref{WhitneyTwist} are related via a Whitney twist on vertices $v$ and $w$. Therefore they  are 2-isomorphic and have the same Tutte polynomial
$$T(x,y) = x + 3 x^2 + 4 x^3 + 4 x^4 + 3 x^5 + x^6 + y + 4 x y + 5 x^2 y + 
 4 x^3 y + 2 x^4 y + 2 y^2 + 3 x y^2 + x^2 y^2 + y^3.$$
 \begin{figure}[h]
  \centering
     \includegraphics[width=0.7\textwidth]{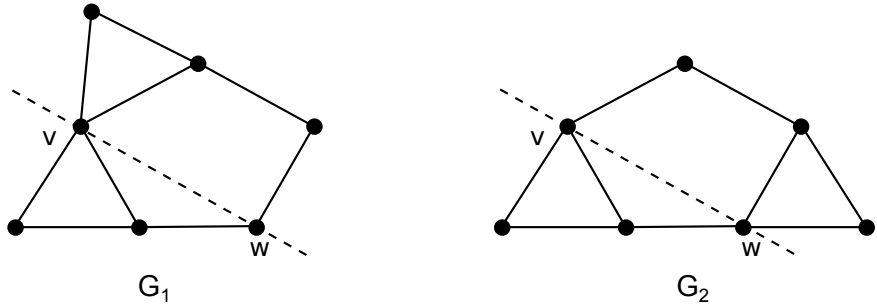}
   \caption{Two graphs in the same 2-isomorphism class.}
    \label{WhitneyTwist}
\end{figure}
However their chromatic homology over $\A_3$ differs already in the zeroth homology group:

\parbox{0.45\textwidth}{\begin{align*}
&H^{0, 10}_{\A_3}(G_1) = \Z^{11}\\
&H^{0, 10}_{\A_3}(G_2) = \Z^{10} 
 \end{align*}}
\parbox{0.45\textwidth}{\begin{align*}
&H^{1, 10}_{\A_3}(G_1) = \Z^{5} \oplus \Z_3^8\\
&H^{1, 10}_{\A_3}(G_2) = \Z^{4} \oplus \Z_3^9
 \end{align*}}
\end{exa}


\bibliography{paperbib1}{}
\bibliographystyle {amsalpha}
\end{document}